\documentclass[11pt]{article}

\usepackage{amsmath}

\usepackage{breqn} 

\usepackage{tikz}

\usepackage[hyperfootnotes=false]{hyperref} 
\usepackage{xcolor}
\definecolor{Bath}{RGB}{0,59,209}
\hypersetup{
colorlinks=true,
linkcolor=Bath,
citecolor=Bath,
urlcolor=Bath
}

\usepackage{amsmath,amsfonts,amssymb,amsthm,epsfig,epstopdf,url,array}
\usepackage{bm}
\usepackage{bbm}
\usepackage{float}
\usepackage{mathtools}
\usepackage{esint}
\usepackage{cases}
\usepackage{verbatim}
\usepackage[font=small,labelfont=bf]{caption}

\usepackage[nottoc,notlot,notlof]{tocbibind}

\usepackage{enumitem}

\usepackage[toc,page]{appendix}

\usepackage{graphicx}
\makeatletter
\newcommand*\bigcdot{\mathpalette\bigcdot@{.5}}
\newcommand*\bigcdot@[2]{\mathbin{\vcenter{\hbox{\scalebox{#2}{$\m@th#1\bullet$}}}}}
\makeatother

\usepackage[a4paper,bindingoffset=0.2in,%
            left=1in,right=1in,top=1.5in,bottom=1.5in,%
            footskip=.25in]{geometry}

\newcommand{\norm}[1]{\left\lVert#1\right\rVert}
\newcommand{\normm}[1]{\lVert#1\rVert}

\newcommand{\sgn}{\textnormal{sgn}}

\usepackage[symbol]{footmisc}

\newtheorem{thm}{Theorem}

\newtheorem{lem}[thm]{Lemma}
\newtheorem{rem}[thm]{Remark}
\newtheorem{cor}[thm]{Corollary}
\newtheorem{defn}[thm]{Definition}
\newtheorem{prop}[thm]{Proposition}

\providecommand{\keywords}[1]
{
  \small	
  \textbf{\textit{Keywords---}} #1
}

\providecommand{\msc}[1]
{
  \small	
  \textbf{\textit{MSC---}} #1
}

\usepackage{setspace}
\begin{document}

\title{\vspace{-3cm}Weighted $\infty$-Willmore Spheres\vspace{-0.2cm}}
\author{
Ed Gallagher \thanks{Department of Mathematical Sciences, University of Bath, Bath, BA2 7AY, UK. Email: \href{redg22@bath.ac.uk}{redg22@bath.ac.uk}}
\and
Roger Moser \thanks{Department of Mathematical Sciences, University of Bath, Bath, BA2 7AY, UK. Email: \href{rm257@bath.ac.uk}{rm257@bath.ac.uk}}
}
\date{ \vspace{-1cm}}

\maketitle


\begin{abstract}
\noindent On the two-sphere $\Sigma$, we consider the problem of minimising among suitable immersions $f \,\colon \Sigma \rightarrow \mathbb{R}^3$ the weighted $L^\infty$ norm of the mean curvature $H$, with weighting given by a prescribed ambient function $\xi$, subject to a fixed surface area constraint. We show that, under a low-energy assumption which prevents topological issues from arising, solutions of this problem and also a more general set of ``pseudo-minimiser'' surfaces must satisfy a second-order PDE system obtained as the limit as $p \rightarrow \infty$ of the Euler-Lagrange equations for the approximating $L^p$ problems. This system gives some information about the geometric behaviour of the surfaces, and in particular implies that their mean curvature takes on at most three values: $H \in \{ \pm \normm{\xi H}_{L^\infty} \}$ away from the nodal set of the PDE system, and $H = 0$ on the nodal set (if it is non-empty).
\end{abstract}


\msc{35A15, 49Q10, 53A05, 53C42, 58E12}

\section{Introduction}\label{sec:intro}

\noindent Functionals involving the mean curvature of surfaces have been studied extensively, both as interesting objects in their own right and as important tools in applications to other areas. The most famous such functional is the \textit{Willmore energy}, defined for an immersion $f : \Sigma \rightarrow \mathbb{R}^n$ of a surface $\Sigma$ with $L^2$ mean curvature $H$ and induced surface measure $\mu$ by
\[
\mathcal{W}[f] := \int_\Sigma H^2 \, \mathrm{d}\mu = \normm{H}_{L^2}^2.
\]
Note that here and throughout this paper we use the convention that the mean curvature is given by the \textit{mean} of the principal curvatures $(\kappa_1 + \kappa_2)/2$ rather than the sum of them. The Willmore energy takes its name from Willmore, who conjectured \cite{willmore1965note} that the inequality $\mathcal{W}[f] \geq 2 \pi^2$ holds for every immersion $f: \mathbb{T}^2 \rightarrow \mathbb{R}^3$ of the torus $\mathbb{T}^2$ into three-dimensional Euclidean space. Nearly 50 years after the conjecture was made, it was confirmed in the celebrated paper of Marques \& Neves \cite{marques2014min}. The Willmore energy and its critical points (known as ``Willmore surfaces''), in particular minimisers, have been studied extensively (see e.g. \cite{li1982new}, \cite{li2002willmore}, \cite{riviere2008analysis}, \cite{kuwert2004removability}, \cite{bryant1984duality}, \cite{simon1993existence}, \cite{bohle2008constrained}, \cite{bernard2014energy}, \cite{musso1990willmore}, \cite{kuwert2001willmore}, \cite{castro2001willmore}, though this list does not pretend to be exhaustive). \\

\noindent To give some other examples of mean-curvature-based functionals, minimisers of
\[
\int_\Sigma (H - c_0)^2 \, \mathrm{d}\mu,
\]
where fixed surface area and volume constraints may be imposed and $c_0$ is a so-called ``spontaneous curvature'' constant, have been proposed \cite{CANHAM1970} \cite{Helfrich1973elastic} \cite{zhong1987instability} to represent the shape of red blood cells and more generally elastic membranes. In \cite{gruber2019variation} and \cite{gruber2019curvature} the $L^p$ analogue of the Willmore energy is studied, again under surface area and volume constraints which ensure the problem is meaningful. In \cite{bernard2019uniform}, critical points of functionals of the form
\[
\int_\Sigma F(H^2) \, \mathrm{d}\mu
\]
are studied for a wide class of functions $F$, and it is shown that under certain assumptions such critical points must be smooth. The functional
\[
\tilde{\mathcal{W}}_p [f] = \int_\Sigma \left( \frac{1}{4} + H^2 \right)^{\frac{p}{2}}, 
\]
is investigated in \cite{kuwert2015two}; here an area restriction is in a sense built into the functional, which is immediately seen to be bounded from above and below by
\[
4^{-\frac{p}{2}} \mathcal{A}[f] \leq \tilde{\mathcal{W}}_p [f] \leq 2^{-\frac{p}{2}} \mathcal{A}[f] + 2^{\frac{p}{2}} \int_\Sigma H^2 \, \mathrm{d}\mu.
\]
The results of \cite{kuwert2015two} are discussed in more detail in Section \ref{sec:approx} as they are used to obtain the results of this paper. \\

\noindent To the best of our knowledge, although as shown by the preceding examples the $L^p$ mean curvature and similar functionals have been investigated in the case where $p$ is finite, no attention has previously been given to what happens when $p = \infty$, which it is the purpose of this paper to consider. More specifically, we consider the functional given by the weighted $L^\infty$ norm of the mean curvature of immersions in $\mathbb{R}^3$ and we investigate how its minimisers, as well as a slightly wider class of ``pseudominimiser'' immersions, behave. It is perhaps unsurprising that this problem has not been looked at before, as variational problems in $L^\infty$ are fundamentally and significantly different to variational problems in finite $L^p$ spaces-- in particular, the standard variational approach of computing the Euler-Lagrange equations cannot possibly work as the $L^\infty$ norm is not differentiable, and it is not clear how to define a critical point of the $L^\infty$ norm (although local/global minimisers still make sense). \\

\noindent First order variational problems in $L^\infty$ and their associated PDEs were first studied by Aronsson in the 1960s in the series of papers \cite{aronsson1965minimization}, \cite{aronsson1966minimization}, \cite{aronsson1969minimization}, and the first-order case is now well-developed, but this is less relevant to curvature functionals because curvature is a property which depends on up to second derivatives. In contrast, second order $L^\infty$ variational problems are inherently more complicated than the first order problems and therefore are less well understood, however they have been considered previously such as in \cite{katzourakis2019existence} which considers variational problems for second order $L^\infty$ functionals of a given form. \\
\noindent Although there are examples where $L^\infty$ variational problems give rise to a connection between minimisers and a system of differential equations, they almost exclusively focus on problems without a geometric background; the rare exceptions are fundamentally different to the problem of minimising the mean curvature and so the results of this paper represent a novel and significant development of the theory. The relevant previous literature which combines geometric analysis and calculus of variations in $L^\infty$ includes \cite{moser2021structure} and its generalisation \cite{gallagher2022infty}, which are concerned with the minimisation of the curvature of curves obeying boundary and length constraints, and \cite{moser2012minimizers} and \cite{sakellaris2017minimization} which consider $L^\infty$ minimisation problems for the scalar (Gaussian) curvature. The problem for curves is in a sense the one-dimensional analogue of the (two-dimensional) problem we investigate here and so it is naturally simpler to work with, and moreover we consider a different type of weighting to that seen in \cite{moser2021structure}. The scalar curvature of course behaves very differently to the mean curvature as e.g. the latter quantity is extrinsic and the former is intrinsic. \\
 
\noindent Specifically, the setup of the problem we consider is the following: let $\Sigma$ be the topological two-sphere and take $\mathcal{G}$ to be the set of admissible immersions of $\Sigma$ into $\mathbb{R}^3$: those which are sufficiently regular and have prescribed surface area $A$, denoted by the set
\[
\mathcal{G} = \left\{ f \in \bigcap_{q < \infty} W^{2,q}(\Sigma,\mathbb{R}^3) \,:\, H \in L^\infty, \, \mathcal{A} [f] = A \right\}.
\]
Let $\xi: \mathbb{R}^3 \rightarrow \mathbb{R}$ be a function satisfying the following three conditions: $\xi$ must be sufficiently regular in the sense that
\begin{equation}\label{cond:alph1}
\xi \in W^{2,\infty}_{\textnormal{loc}}(\mathbb{R}^3),
\end{equation}
$\xi$ must be bounded from below by one,
\begin{equation}\label{cond:alph2}
\xi \geq 1,
\end{equation}
and $\xi$ must tend to infinity as $\vert x \vert$ does:
\begin{equation}\label{cond:alph3}
\lim_{\vert x \vert \rightarrow \infty} \xi(x) = \infty.
\end{equation}
\noindent These conditions allow for a wide and general range of weight functions $\xi$. \\
\noindent We seek to solve the following minimisation problem:
\begin{equation}
\begin{dcases}\label{problem}
\text{Minimise the functional } \mathcal{M}_\infty [f] := \normm{\xi H}_{L^\infty}\\
\text{over the set of immersions } \mathcal{G}.
\end{dcases}
\end{equation}
The prescribed surface area $A > 0$ can be any given positive constant, and the area functional $\mathcal{A}$ of a surface is given by
\[
\mathcal{A}[f] = \int_\Sigma 1 \, \mathrm{d}\mu.
\]
The norm $L^\infty$ in \eqref{problem} refers to the $L^\infty$ norm taken with respect to the surface measure $\mu$ induced by the immersion $f$ itself. The Sobolev space $W^{2,q}(\Sigma,\mathbb{R}^3)$, where $1 \leq q \leq \infty$, can be defined e.g. as the set of $f: \Sigma \rightarrow \mathbb{R}^3$ such that the differential of $f$ has rank $2$ at every point (i.e. $f$ is an immersion) and at every point $x \in \Sigma$, $f$ can be expressed as a graph $(x,y,u(x,y))$ over some ball $B$ contained in the tangent plane $T_{f(x)}f(\Sigma)$ with $u \in W^{2,q}(B)$. It is more common in the literature to not require that elements of $W^{2,q}(\Sigma,\mathbb{R}^3)$ be immersions, although for the purposes of this paper we may freely assume that all maps $f$ we work with are immersions and so this point is inconsequential. \\

\noindent The inclusion of the weight function $\xi$ is essential, or else it is extremely quick to solve the problem using relatively simple machinery: with no weight, i.e. in the case $\xi \equiv 1$, the inequality $H^2 \geq K$ and the Gauss-Bonnet theorem together imply that the round sphere is the unique minimiser (up to translation). The surface area constraint is necessary to ensure the problem is non-degenerate (else uniform scalings of a given surface would result in immersions with $\normm{H}_{L^\infty}$ arbitrarily close to zero), as is condition \eqref{cond:alph3} on the weight function (else it would be possible for a minimising sequence to ``translate to $\infty$'' and fail to converge). The presence of condition \eqref{cond:alph3} may appear slightly artificial; in fact it is not a necessary condition for minimisers to exist, and there are other conditions it could be replaced by which would also ensure existence. However, we have included it in its given form here to prevent our proofs from becoming overly bogged down with technicalities and highly involved arguments that distract from the main story of the paper. We believe that any reasonable hypothesis which means it is not advantageous for a minimising sequence to ``translate to $\infty$'' could replace condition \eqref{cond:alph3} and allow our results to remain valid; for example, one could assume that $\xi(x)$ is an increasing function of $\vert x \vert$ outside of some compact subset of $\mathbb{R}^3$, or that $\xi$ is periodic. \\

\noindent Note that we do not consider the analogous $L^\infty$ problem with the enclosed volume of the immersions controlled instead of their area. Although this seems like a natural problem to consider, it is in fact degenerate, as \cite{ferone2016maximal} shows by constructing a smooth embedding of the sphere with arbitrarily small volume for given maximum mean curvature; hence by scaling, arbitrarily small $\normm{H}_{L^\infty}$ for given volume. \\

\noindent As well as minimisers of $\mathcal{M}_\infty$, the results of this paper naturally suggest that we also consider the following class of surfaces, which we call ``$\infty$-Willmore spheres'':
\begin{defn}\label{def:infWillmoresurface}
Fix a weight function $\xi: \mathbb{R}^3 \rightarrow \mathbb{R}_{\geq 1}$ satisfying \eqref{cond:alph1}--\eqref{cond:alph3} and a prescribed surface area constant $A > 0$. An immersion $f \in \mathcal{G}$ is called an $\infty$-Willmore sphere if there exists a constant $M > 0$ such that the inequality
\[
\mathcal{M}_\infty [f] \leq \mathcal{M}_\infty[f_0] + \frac{M}{2A} \int_\Sigma \textnormal{dist}(f_0(x),f(\Sigma))^2 \, \mathrm{d}\mu_0
\]
is satisfied for all $f_0 \in \mathcal{G}$. \\
Here, $\textnormal{dist}(f_0(x),f(\Sigma))$ denotes the (Euclidean, $\mathbb{R}^3$) distance between the point $f_0(x)$ and the set $f(\Sigma)$, and $\mu_0$ denotes the surface measure induced by the immersion $f_0$. 
\end{defn}
\noindent It is clear from Definition \ref{def:infWillmoresurface} that minimisers of our problem are automatically $\infty$-Willmore spheres, but the definition also allows for a wider class of ``pseudo-minimiser'' immersions which are not minimal. We shall see in Section \ref{sec:limitsystem} that under certain natural conditions, $\infty$-Willmore spheres (minimisers, in fact) are guaranteed to exist. Because the definition of an $\infty$-Willmore sphere only depends on $f$ via its image $\text{Im}(f)$ and not its specific parametrisation, we may slightly abuse notation by (implicitly) identifying an immersion with its image.\\

\noindent The main result of this paper is the following, concerning the structure and geometry of $\infty$-Willmore spheres:
\begin{thm}\label{thm:mainthm}
Let $\xi$ be a given weight function satisfying \eqref{cond:alph1}$\,$--$\,$\eqref{cond:alph3} and let $f \in \mathcal{G}$ be an $\infty$-Willmore sphere. Assume that the low-energy inequality $\normm{\xi H}_{L^\infty} A^{\frac{1}{2}} < \sqrt{8\pi}$ is satisfied. Then the following statements hold: \\
1) There exist a function
\[
w \in \bigcap_{q < \infty} W^{2,q}(\Sigma) \backslash \{ 0 \},
\]
a function $Q \in \bigcap_{q < \infty} L^q(\Sigma)$, and a number $\lambda \in \mathbb{R}$ such that the system of equations
\begin{align}
   \frac{1}{2} \Delta w + Q w &= \lambda H, \label{eq:main1}\\
   hw &= \vert w \vert \xi H, \label{eq:main2}
\end{align}
is satisfied almost everywhere on $\Sigma$, where $\Delta$ denotes the Laplace-Beltrami operator on the surface $f(\Sigma)$, and $h = \normm{\xi H}_{L^\infty} > 0$ is the weighted $L^\infty$ norm of the mean curvature. \\
2) The weighted mean curvature $\xi(f(x)) H(x)$ of $f$ takes on only three values up to null sets: where $w$ is positive, $\xi H = h$; where $w$ is zero, $\xi H = 0$; and where $w$ is negative, $\xi H = -h$.
\end{thm}
\noindent We also obtain sufficient conditions for the existence of an $\infty$-Willmore sphere (in particular a minimiser):
\begin{prop}\label{prop:sufexist}
Suppose we are given a weight function $\xi$ satisfying \eqref{cond:alph1}--\eqref{cond:alph3} and a surface area constraint $A > 0$. Assume that there exists an immersion $f \in \mathcal{G}$ satisfying the low-energy inequality
\[
\normm{\xi H}_{L^\infty} A^{\frac{1}{2}} < \sqrt{8 \pi}.
\]
Then there exists a solution of Problem \eqref{problem}.
\end{prop}


\noindent Our results are derived by considering the analogous $L^p$ problem to our $L^\infty$ problem for $p < \infty$, computing the $L^p$ Euler-Lagrange equations, and letting $p$ tend to $\infty$. This method is known as the $L^p$ approximation of the problem. It is because of this approach that we require a strict $8 \pi$ bound on $\normm{\xi H}_{L^\infty} A^{\frac{1}{2}}$ in Theorem \ref{thm:mainthm}; as will be discussed more in Section \ref{sec:approx}, we will obtain $\infty$-Willmore spheres as a sequential limit of immersions from the $L^p$ approximation, and the energy bound is required to ensure the necessary compactness is present by preventing topological problems from arising in the limit. Replacing the assumption in Proposition \ref{prop:sufexist} is in theory possible, for example with some conditions involving a positive lower bound on the injectivity radius, although this would be extremely unwieldy: our assumption is used not only to ensure the existence of our $\infty$-Willmore spheres but also of the sequence of immersions which approximate them, and moreover our assumption is specifically chosen to interact well with the results of \cite{kuwert2015two} which we rely on heavily in our proofs. \\
Although the $L^p$ approximation is widely known and considered a relatively standard technique within the field of calculus of variations in $L^\infty$, its application to geometric problems is non-standard and examples in the literature are sparse, some cases we note here being the aforementioned papers \cite{moser2021structure} \cite{gallagher2022infty} \cite{moser2012minimizers} \cite{sakellaris2017minimization}. As far as the authors know, the way in which it is used in this paper (for sequences of immersions of surfaces) is novel. \\

\noindent This paper is structured as follows. In Section \ref{sec:approx}, we consider the $L^p$ problem with finite $p$ which corresponds to our $L^\infty$ problem and we derive its Euler-Lagrange equations. For technical reasons involving regularity, we approximate the $L^p$ problem itself by another problem, for which we also consider the existence of minimisers and derive the Euler-Lagrange equations. In Section \ref{sec:compactness}, we use our approximation of the $L^p$ problem to show that there exist minimisers of the $L^p$ problem which are sufficiently regular for our purposes. We demonstrate the necessary compactness to take the limit as $p \rightarrow \infty$ of a sequence $(f_p)$ of $L^p$-minimisers and obtain not only a limiting $\infty$-Willmore sphere, but also equations which govern its behaviour, derived from the limit of the $L^p$ Euler-Lagrange equations. This begins the proof of Theorem \ref{thm:mainthm}.
In Section \ref{sec:limitsystem}, we use the limiting system \eqref{eq:main1}--\eqref{eq:main2} obtained in Section \ref{sec:compactness} to obtain some information about how the mean curvature of $\infty$-Willmore spheres behaves, finishing the proof of Theorem \ref{thm:mainthm}. We conclude by proving Proposition \ref{prop:sufexist}, exhibiting natural and relatively loose conditions on our weight function and area constraint under which $\infty$-Willmore spheres are guaranteed to exist. 


\newpage
\section{The Penalised $L^p$ Problem}\label{sec:approx}


\subsection{The Penalisation Term}\label{subsec:penterm}

\noindent Because we obtain our $\infty$-Willmore spheres as limits of appropriate $L^p$ minimisers, to ensure that Theorem \ref{thm:mainthm} applies to \textit{all} $\infty$-Willmore spheres we must ensure that all such surfaces will arise as a sequential limit in this way. To do this we add a penalisation term to our analysis, as was done in \cite{moser2021structure} and \cite{gallagher2022infty}. Although Definition \ref{def:infWillmoresurface} seems to suggest that we involve the distance function in our penalisation term, its potential lack of differentiability could cause technical problems, and so we will instead consider a regularisation of the distance function as follows. Given an immersion $f_0 \in \mathcal{G}$ to penalise against, let $d : \mathbb{R}^3 \rightarrow \mathbb{R}$ be a smooth function satisfying the inequalities
\begin{equation}\label{eq:distineqs}
\frac{1}{C} \textnormal{dist}(y,f_0(\Sigma)) \leq d(y) \leq C \textnormal{dist}(y,f_0(\Sigma)), \quad \quad \vert D d (y) \vert\, \leq C
\end{equation}
for all $y \in \mathbb{R}^3$ and some constant $C$ independent of $f_0$ (for an explicit construction of such a function see e.g. the book of Stein \cite[Chapter 6.2.1]{stein1970singular}). 

\noindent Define the ``penalisation functional'' by
\[
\mathcal{P}^\sigma [ f ; f_0 ] = \frac{\sigma}{2A} \int_\Sigma d(f(x))^2 \, \mathrm{d}\mu
\] 
where $\sigma > 0$ is a given real number. In particular, the penalisation functional is independent of the parametrisation of $f$.  \\

\noindent In light of the preceding discussion, we consider the functional
\[
\mathcal{M}_\infty^{\sigma} [f ; f_0] = \normm{\xi H}_{L^\infty} + \mathcal{P}^\sigma [ f ; f_0 ]
\]
over immersions $f \in \mathcal{G}$ into three-dimensional Euclidean space, where $f_0$ is any given (immersion of an) $\infty$-Willmore sphere. 

\subsection{Minimisers of the $L^p$ functional $\mathcal{M}_p^{\sigma} [ \,\bigcdot\, ; f_0 ]$}\label{subsec:Lpmin}

\noindent The $L^p$ approximating functionals for $\mathcal{M}_\infty^{\sigma}$ are given by
\[
\mathcal{M}_p^{\sigma} [f ; f_0 ] := A^{- \frac{1}{p}} \normm{\xi H}_{L^p} + \mathcal{P}^\sigma [ f ; f_0 ],
\]
with the admissible set of immersions given by $\{ f \in W^{2,p}(\Sigma, \mathbb{R}^3) \,:\, \mathcal{A}[f] = A  \}$. The factor of $A^{- \frac{1}{p}}$ in front of the $L^p$ norm is included to more conveniently renormalise the Euler-Lagrange equations we will obtain. \\

\noindent The standard approach to show existence of minimisers is the direct method, which allows us to take a minimising sequence for the functional $\mathcal{M}_p^{\sigma} [\,\bigcdot\, ; f_0 ]$ and obtain convergence in an appropriate sense to a minimiser $f_p$ (see e.g. the book \cite{rindler2018calculus} for an introduction). However, when trying to show coercivity we may run into problems because it is not possible to bound the $W^{2,p}$ norm of $f$ in terms of the $L^p$ norm of $H$ without also having a lower bound on the graph radius of $f$, and this is not something we possess a priori. To get around this we employ the results of \cite{kuwert2015two}, which are invaluable to the work of this paper.

\begin{thm}[Theorems 1.1 and 1.2 of \cite{kuwert2015two}]\label{thm:KLL}
Let $\Sigma$ be a closed surface and let the number $p \in (2,\infty)$ be fixed. Consider a sequence of immersions $(f_k) \subset W^{2,p}(\Sigma,\mathbb{R}^n)$ which is assumed to satisfy the condition $0 \in f_k(\Sigma)$ for all $k \in \mathbb{N}$ and the inequalities
\begin{equation}\label{eq:Wpineq}
\left( \int_\Sigma \left( \frac{1}{4} + H_k^2 \right)^{\frac{p}{2}} \, \mathrm{d}\mu_k \right)^{\frac{1}{p}} \leq C
\end{equation}
and
\begin{equation}\label{eq:Willmoreineq}
\liminf_k \int_\Sigma \vert H_k \vert^2 \, \mathrm{d}\mu_k < 8 \pi,
\end{equation}
where $H_k$ and $\mu_k$ denote the mean curvature and surface measure of $f_k$ respectively. \\
After composition with diffeomorphisms $\varphi_k \in C^\infty(\Sigma,\Sigma)$ and taking a subsequence (not explicitly labelled), the sequence $(f_k \circ \varphi_k)$ converges weakly in $W^{2,p}(\Sigma,\mathbb{R}^n)$ to an immersion $f \in W^{2,p}(\Sigma,\mathbb{R}^n)$ with mean curvature $H$ and surface measure $\mu$ satisfying the inequality
\[
\left(\int_\Sigma \left( \frac{1}{4} + H^2 \right)^{\frac{p}{2}} \, \mathrm{d}\mu \right)^{\frac{1}{p}} \leq \liminf_k \left( \int_\Sigma \left( \frac{1}{4} + H_k^2 \right)^{\frac{p}{2}} \, \mathrm{d}\mu_k \right)^{\frac{1}{p}} .
\]
Defining the real number $r_k$ as the infimum across all points $x \in \Sigma$ of the maximal radius over which $f_k$ can be represented as a graph over the tangent plane $T_{f_k(x)}f_k(\Sigma)$, the inequality
\begin{equation}\label{eq:graphrad}
\liminf_k r_k > 0
\end{equation}
holds. \\
If $f$ is a critical point of the functional 
\[
f \mapsto \left(\int_\Sigma \left( \frac{1}{4} + H^2 \right)^{\frac{p}{2}} \, \mathrm{d}\mu \right)^{\frac{1}{p}}
\]
then $f$ is real analytic.
\end{thm}

\noindent Inequality \eqref{eq:graphrad} is not explicitly formulated in the main theorems of \cite{kuwert2015two}, but can be found in the proof of Theorem 1.2, cf. inequality (4.6) there. \\
\noindent For convenience of notation, when we apply Theorem \ref{thm:KLL} in this paper we will not make the subsequential convergence or the composition with diffeomorphisms explicit. \\

\noindent The bound on the Willmore energy \eqref{eq:Willmoreineq} is necessary to prevent topological problems from arising, as illustrated by the following example \cite{kuwert2015two}: take two round spheres connected by a catenoidal neck ($H \equiv 0$) which shrinks and becomes thinner as $n \rightarrow \infty$ with width tending to zero. In this example we do not have convergence to a limiting immersion of $\Sigma$; the problem is that the the surface, topologically a sphere, is trying to `split up' in the limit into an immersion of two spheres. This example shows that the $8 \pi$ bound is sharp. Another demonstration of the importance of inequality \eqref{eq:Willmoreineq} can be found in \cite{blatt2009singular}, which constructs for arbitrary $\varepsilon > 0$ a surface with Willmore energy $8 \pi + \varepsilon$ whose Willmore flow develops singularities (possibly after infinite time). To see how inequality \eqref{eq:Willmoreineq} prevents topological issues in our sequence of surfaces, we note that for any immersion of a closed surface the inequality
\[
\int_\Sigma H^2 \, \mathrm{d}\mu \geq 4 \pi
\]
holds (this fact, proved by Willmore \cite{willmore1965note} \cite{willmore1968note}, comes from elementary arguments; for details see e.g. \cite{topping2000towards}). Hence a uniform Willmore energy bound of less than $8 \pi$ prevents our sequence of surfaces from trying to create multiple connected components, as each one would contribute at least $4 \pi$ to the energy. Note also that by a result from \cite{li1982new}, any immersion of a compact surface which is not embedded must have a Willmore energy of \textit{at least} $8\pi$, so the convergence result of Theorem \ref{thm:KLL} will not apply to a sequence of non-embedded surfaces. \\

\noindent We expect for surfaces of genus at least one that Proposition \ref{prop:sufexist} will not apply because of the failure of the inequality $\mathcal{M}_\infty[f] A^{\frac{1}{2}} < \sqrt{8\pi}$ (although we do not have a full proof that the bound must fail). It is this inequality that allows us to approximate our $\infty$-Willmore spheres by $\mathcal{M}_p^\sigma$ minimisers: whenever it is satisfied, H\"{o}lder's inequality implies that the Willmore energy of $f$ must be beneath $8 \pi$, and this allows us to use the $L^p$ approximation without the risk of the topology of the immersions changing as $p$ goes to infinity. On the other hand, when the above $\mathcal{M}_\infty[f]$ inequality is not satisfied, we can no longer guarantee that $f$ has Willmore energy beneath $8 \pi$ and we cannot apply Theorem \ref{thm:KLL}. Note that the problem when we attempt to apply our methods to surfaces of positive genus is not that the Willmore energy is too large for any immersion-- indeed, as mentioned in \cite{kuwert2015two}, it is possible to attain a Willmore energy strictly below $8\pi$ for a surface of any genus-- but rather the issue is that we cannot ensure (and do not believe in general) that either our $\mathcal{M}_\infty$ minimisers, or their approximating $\mathcal{M}_p$ minimisers when $p$ is large, will have Willmore energy less than $8 \pi$. Hence when we attempt to use the $L^p$ approximation, we cannot guarantee that the Willmore energy of the sequence of $\mathcal{M}_p$ minimisers remains below $8\pi$ as $p \rightarrow \infty$, even if the Willmore energy is low when $p=2$.  \\

\noindent Although our approach does not work for surfaces of positive genus, there is no strong reason to believe that a minimising sequence should fail to converge in the appropriate sense in this case. To give examples demonstrating cases where the $L^\infty$ bound fails, explicit calculations show that on any torus of revolution it does not hold, including the torus with central/tubular radii in the ratio $\sqrt{2}:1$ which achieves the infimum Willmore energy $2\pi^2$ among tori in $\mathbb{R}^3$. In \cite{scharrer2021embedded}, a family of embedded ``Delaunay Tori'' of piecewise constant mean curvature are constructed by joining together a nodoid (constant negative $H$) and an unduloid (constant positive $H$) in such a way to give rise to a $W^{2,\infty}$ surface. At one end of the family the Willmore energy is below $8\pi$, but the quantity $\normm{H}_{L^\infty}^2 A$ stays above $8 \pi$. \\

\noindent It is possible to prove the existence of minimisers of $\mathcal{M}_p^{\sigma} [ \,\bigcdot\, ; f_0 ]$ for any given $p \in (2,\infty)$ using Theorem \ref{thm:KLL}. For our purposes however it is desirable to have minimisers whose regularity we have more control over, and for this we need to construct the minimisers in a specific way. This is the focus of Section \ref{subsec:epsapprox}, in which we provide more detailed discussion on the how the $\mathcal{M}_p^{\sigma} [ \,\bigcdot\, ; f_0 ]$ minimisers are constructed. In Section \ref{subsec:epsconv} we prove rigorously that this construction does in fact lead to minimisers of $\mathcal{M}_p^{\sigma} [ \,\bigcdot\, ; f_0 ]$ of specific regularity. \\

\noindent We briefly remark on how the convergence of immersions in Theorem \ref{thm:KLL} interacts with the convergence of the mean curvature of the immersions.
\begin{rem}\label{rem:meancurvconv}
For a sequence $(f_k)$ converging to a limiting immersion $f$ in the weak $W^{2,p}$ sense of Theorem \ref{thm:KLL}, the sequence of mean curvatures $(H_k)$ converges weakly in $L^p$ to $H$, the mean curvature of the limit $f$.
\end{rem}
\noindent The validity of this remark can be seen via the following argument: over graphs $(x,y,u(x,y))$, the expression for the mean curvature
\[
H(x,y) = \frac{(1 + u_y^2)u_{xx} - 2 u_x u_y u_{xy} + (1 + u_x^2)u_{yy}}{2\left(1 + u_x^2 + u_y^2 \right)^{\frac{3}{2}}}
\]
is linear in the second derivatives of $u$, so the weak $W^{2,p}$ convergence $f_k \rightharpoonup f$ (which implies strong $C^{1,\beta}$ convergence) combined with the graph radius inequality \eqref{eq:graphrad} of Theorem \ref{thm:KLL} suffices to show the weak $L^p$ convergence of $H_k \rightharpoonup H$. \\

\noindent We now focus on computing the Euler-Lagrange equations associated to the functional $\mathcal{M}_p^{\sigma} [ \,\bigcdot\, ; f_0 ]$. To do this we need to work out the first variations of the various terms in $\mathcal{M}_p^{\sigma} [\,\bigcdot\, ; f_0]$. We follow closely the approach from \cite{willmore1982total}. The calculations are simplified considerably by the observation that all the terms in the functional $\mathcal{M}_p^{\sigma} [ f ; f_0 ]$ depend only on the image of $f$ and not its parametrisation, so for such terms it is sufficient to only consider variations normal to $f$. \\

\noindent  The main ingredients of the Euler-Lagrange equations are the following:
\begin{lem}\label{lem:ELcomps}
Let $f: \Sigma \times I \rightarrow \mathbb{R}^3$, $I = (-\varepsilon, \varepsilon)$ be a normal variation with velocity field $\partial_\varepsilon \bigr|_{\varepsilon = 0} f = \phi \nu$ where $\phi$ is a given scalar test function. Then the following equations hold:
\begin{align}
\partial_\varepsilon \bigr|_{\varepsilon = 0} \sqrt{\det g} &= - 2 \phi H \sqrt{\det g} \label{eq:firstvar1}, \\
\partial_\varepsilon \bigr|_{\varepsilon = 0}\, H &= \frac{1}{2} \Delta \phi + \phi(2H^2 - K) \label{eq:firstvar2}, \\
\partial_\varepsilon \bigr|_{\varepsilon = 0}\, \xi &= \phi D_\nu \xi \label{eq:firstvar3}, \\
\partial_\varepsilon \bigr|_{\varepsilon = 0}\, d(f(x))^2 &= 2 \phi d D_\nu d. \label{eq:firstvar4}
\end{align}
\end{lem}
\noindent Here we have denoted the unit normal vector field to $f(\Sigma)$ by $\nu$, and in equations \eqref{eq:firstvar3}--\eqref{eq:firstvar4} the operator $D_\nu$ denotes the (Euclidean) directional derivative in the direction of $\nu$. \\

\noindent Equations \eqref{eq:firstvar1}--\eqref{eq:firstvar2} can be found in the book \cite{willmore1982total}, while equations \eqref{eq:firstvar3}--\eqref{eq:firstvar4} come from direct computation. \\

\noindent The following proposition, which expresses the Euler-Lagrange equations for the functional $\mathcal{M}_p^{\sigma}[\,\bigcdot\, ; f_0]$, is derived from Lemma \ref{lem:ELcomps} via standard computations which we omit here.
\begin{prop}\label{prop:pELeqs}
Let the immersion $f_p$ be a critical point of $\mathcal{M}_p^{\sigma}[\,\bigcdot\, ; f_0]$. Let $\Delta_p$ denote the Laplace-Beltrami operator on $f_p(\Sigma)$ and denote by $\xi_p$ the composition $\xi \circ f_p$ of the weight function with $f_p$ (that is, $\xi_p$ is the weight function evaluated on the surface $f_p(\Sigma)$). Let $H_p$ and $K_p$ denote the mean and Gaussian curvatures of $f_p$ respectively. Set
\[
h_p = A^{-\frac{1}{p}} \normm{\xi_p H_p}_{L^p}
\]
and define the functions
\[
w_p = h_p^{1-p} \xi_p^p \vert H_p \vert^{p-2} H_p
\]
and
\[
Q_p = 2 \frac{p-1}{p} H_p^2 - K_p + \frac{H_p D_{\nu_p} \xi_p}{\xi_p}.
\]
Let $d_p(x) = d(f_p(x))$ where $d$ is any given regularisation of the distance function from $f_0(\Sigma)$, as detailed in Section \ref{subsec:penterm}. Then the Euler-Lagrange equations for $f_p$ are given by
\begin{equation}\label{eq:pEL}
\frac{1}{2} \Delta_p w_p + Q_p w_p = \lambda_p H_p - \sigma d_p D_{\nu_p} d_p + \sigma d_p^2 H_p
\end{equation}
where the equation is to be interpreted in the weak sense, and the real number $\lambda_p$ is the Lagrange multiplier arising to account for the area constraint $\mathcal{A}[f_p] = A$.
\end{prop}


\subsection{Approximation of $\mathcal{M}_p^{\sigma} [ \,\bigcdot\, ; f_0 ]$ Minimisers}\label{subsec:epsapprox}

\noindent In the previous subsection we derived the Euler-Lagrange equation \eqref{eq:pEL} governing critical points (in particular minimisers) of $\mathcal{M}_p^{\sigma} [ \,\bigcdot\, ; f_0 ]$. The presence of the Laplace-Beltrami operator applied to the function $w_p$ in \eqref{eq:pEL} suggests that $w_p$ should somehow possess two weak derivatives, although a priori all we can say is that $w_p$ lies in the space $L^{p/(p-1)}$ (this integrability follows immediately from the definition of $w_p$ and the assumption of finite energy for $\mathcal{M}_p^{\sigma} [ \,\bigcdot\, ; f_0 ]$). For our method of $L^p$ approximation to work, we will need more regularity than this, and so we introduce another functional $\mathcal{M}_p^{\sigma,\varepsilon}[ \,\bigcdot\, ;f_0]$ to approximate the $\mathcal{M}_p^{\sigma} [ \,\bigcdot\, ; f_0 ]$ functional. The regularity of critical points of the former can be established from prior results, and in Section \ref{subsec:epsconv} we show that as we let the parameter $\varepsilon$ tend to zero we have convergence to critical points (minimisers) of $\mathcal{M}_p^{\sigma} [ \,\bigcdot\, ; f_0 ]$ in the appropriate sense to give us the desired regularity of the function $w_p$ from \eqref{eq:pEL}. \\

\noindent Specifically, we define the $\mathcal{M}_p^{\sigma,\varepsilon}[ \,\bigcdot\, ;f_0]$ functional by
\[
\mathcal{M}_p^{\sigma,\varepsilon}[ f ;f_0] := A^{-\frac{1}{p}} \left( \int_\Sigma \left( (\xi H)^2 + \varepsilon \right)^{\frac{p}{2}} \,\mathrm{d}\mu \right)^{\frac{1}{p}} + \mathcal{P}^\sigma[f;f_0],
\]
and the set of admissible immersions is given by 
\[
\{ f \in W^{2,p}(\Sigma, \mathbb{R}^3) \,:\, \mathcal{A}[f] = A  \}.
\]

\noindent To show that minimisers of $\mathcal{M}_p^{\sigma,\varepsilon}[ \,\bigcdot\, ;f_0]$ exist, we appeal to Theorem \ref{thm:KLL}.
\begin{lem}\label{lem:epsminexist}
Let $p \in (2,\infty)$ and $\sigma > 0$ be given and let $f_0$ be any $\infty$-Willmore surface satisfying the low-energy inequality in Theorem \ref{thm:mainthm}. There exists a constant $\tilde{\varepsilon} > 0$ independent of $p$ such that for all $\varepsilon < \tilde{\varepsilon}$, at least one minimiser of the functional $\mathcal{M}_p^{\sigma,\varepsilon}[ \,\bigcdot\, ;f_0]$ exists.
\end{lem}
\begin{proof}
If $f_0$ is a minimiser of $\mathcal{M}_p^{\sigma,\varepsilon}[ \,\bigcdot\, ;f_0]$, the Lemma automatically holds. Else, we apply Theorem \ref{thm:KLL} directly. Although Theorem \ref{thm:KLL} is written with the constant $1/4$ in, there is nothing special about this number and the obvious reformulation of the theorem applies when $1/4$ is replaced by any strictly positive constant (strict positivity is vital here). Let $\varepsilon > 0$ be given and take $(f_k)$ to be a minimising sequence for the functional $\mathcal{M}_p^{\sigma,\varepsilon}[ \,\bigcdot\, ;f_0]$. \\
We first check that the inequalities of Theorem \ref{thm:KLL} are satisfied. From the definition of the $\mathcal{M}_p^{\sigma,\varepsilon}[ \,\bigcdot\, ;f_0]$ functional, it is clear that the reformulation of inequality \eqref{eq:Wpineq}
\[
\left( \int_\Sigma \left( \varepsilon + H_k^2 \right)^{\frac{p}{2}} \, \mathrm{d}\mu_k \right)^{\frac{1}{p}} \leq C
\]
holds. Also, when $k$ is large enough that $\mathcal{M}_p^{\sigma,\varepsilon} [f_k ; f_0 ] < \mathcal{M}_p^{\sigma,\varepsilon} [f_0 ; f_0 ]$ (which must happen eventually since we are assuming $f_0$ is not a minimiser) we have the chain of inequalities
\begin{align*}
\normm{H_k}_{L^2} &\leq \normm{\xi_k H_k}_{L^2} \\
&\leq A^{\frac{1}{2} - \frac{1}{p}} \normm{\xi_k H_k}_{L^p} \\
&\leq A^{\frac{1}{2}} \mathcal{M}_p^\sigma [f_k ; f_0 ] \\
&\leq A^{\frac{1}{2}} \mathcal{M}_p^{\sigma,\varepsilon} [f_k ; f_0 ] \\
&\leq A^{\frac{1}{2}} \mathcal{M}_p^{\sigma,\varepsilon} [f_0 ; f_0 ] \\
&\leq \norm{\sqrt{(\xi_0 H_0)^2 + \varepsilon}\,}_{L^\infty} A^{\frac{1}{2}} 
\end{align*}
having used H\"{o}lder's inequality in the second and final lines and assumption \eqref{cond:alph1} on the weight function in the first line. By the assumptions in the statement of the Lemma the inequality
\[
\normm{\xi H}_{L^\infty} A^{\frac{1}{2}} < \sqrt{8\pi}
\]
is satisfied, and so we can define $\tilde{\varepsilon}$ so that 
\[
\norm{\sqrt{(\xi_0 H_0)^2 + \tilde{\varepsilon}}\,}_{L^\infty}A^{\frac{1}{2}}  = \sqrt{8 \pi}.
\]
For all $\varepsilon < \tilde{\varepsilon}$, it follows that we have the inequality $\normm{H_k}_{L^2} < \sqrt{8 \pi} - \delta$ for some small $\delta > 0$, and upon taking the liminf we find that inequality \eqref{eq:Willmoreineq} holds for the sequence $(f_k)$. Translating every element of the sequence to ensure that the condition $0 \in f_k(\Sigma)$ of Theorem \ref{thm:KLL} is satisfied, we obtain the weak $W^{2,p}$ convergence of $(f_k - f_k(x_0))$ to some limiting immersion where $x_0 \in \Sigma$ can be taken as any point. This means that we have the weak $W^{2,p}$ convergence ``of the shape of $(f_k)$'' and in particular the diameters $\textnormal{diam}(f_k(\Sigma))$ of $f_k(\Sigma)$ converge as a sequence of reals (this comes from the strong $C^0$ convergence). \\ 
\noindent We now argue that in fact, we have convergence of the untranslated sequence $(f_k)$ as well. It is sufficient to show that there exists a ball $B_R(0)$ of some radius $R > 0$ such that infinitely many of the immersions $f_{k}$ intersect $B_R(0)$. We consider two cases: the case where $\sigma \neq 0$, i.e. the penalisation term is present in our functional, and the case where $\sigma = 0$. In both cases we argue by contradiction, assuming that for all $R > 0$ a subsequence (not explicitly labelled) of $(f_{k})$ eventually stays outside of $B_R(0)$. \\
When $\sigma \neq 0$, let $r$ be such that $f_0(\Sigma) \subset B_r(0)$. We know from \eqref{eq:distineqs} and our assumption that $(f_{k})$ `translates to infinity' that for any given $R$ and large enough $k$,
\begin{align*}
\mathcal{M}_p^\sigma [ f_{k} ; f_0 ] &\geq \frac{\sigma}{2AC} \int_\Sigma \textnormal{dist}(f_k(x),f_0(\Sigma))^2 \, \mathrm{d}\mu \\
&\geq \frac{\sigma}{2AC} \int_\Sigma (R - r)^2 \, \mathrm{d}\mu \\
&= \frac{(R - r)^2}{2C},
\end{align*}
and taking $R$ arbitrarily large implies that $\mathcal{M}_p^\sigma [ f_{k} ; f_0 ]$ must tend to $\infty$, contradicting the fact that $(f_{k})$ is a minimising sequence for $\mathcal{M}_p^{\sigma,\varepsilon} [ \,\bigcdot\, ; f_0 ]$. \\
If instead $\sigma = 0$, we make use of assumption \eqref{cond:alph3} on our weight function. By the aforementioned convergence of the diameters of $(f_k)$, we can  take $r > 0$ such that each immersion $f_{k}$ fits inside a ball of radius $r$, with no assumption on where the centres of such balls lie. By condition \eqref{cond:alph3}, we may choose $R$ so large that for all $y_1 \in B_r(0)$ and $y_2 \in \mathbb{R}^3 \backslash B_R(0)$, the inequality
\[
\xi(y_1) \leq 2 \xi(y_2)
\]
is satisfied. This again contradicts the choice of $(f_{k})$ as a minimising sequence for $\mathcal{M}_p^{\sigma,\varepsilon}[\,\bigcdot\, ; f_0]$, because by translating each $f_{k}$ so that it lies inside $B_r(0)$ instead of outside $B_R(0)$ the value of $\mathcal{M}_p^{\sigma,\varepsilon}[f_k ; f_0]$ would decrease by at least a fixed amount. Letting $\xi_{k,1}$ denote the weight function applied to the untranslated immersion (i.e. the one outside $B_R(0)$) and $\xi_{k,2}$ denote the weight function applied to the translated immersion, we see that
\begin{align*}
\int_\Sigma \left( ( \xi_{k,1} H_k )^2 + \varepsilon \right)^{\frac{p}{2}} \, \mathrm{d}\mu_k &\geq \int_\Sigma \left( ( 2\xi_{k,2} H_k \right)^2 + \varepsilon)^{\frac{p}{2}} \, \mathrm{d}\mu_k \\
&\geq \int_\Sigma \left( ( \xi_{k,2} H_k )^2 + \varepsilon \right)^{\frac{p}{2}} + \vert \xi_{k,2} H_k \vert^p  \, \mathrm{d}\mu_k \\
&= \int_\Sigma \left( ( \xi_{k,2} H_k )^2 + \varepsilon \right)^{\frac{p}{2}} \, \mathrm{d}\mu_k + \int_\Sigma \vert \xi_{k,2} H_k \vert^p  \, \mathrm{d}\mu_k \\
&\geq \int_\Sigma \left( ( \xi_{k,2} H_k )^2 + \varepsilon \right)^{\frac{p}{2}} \, \mathrm{d}\mu_k + \int_\Sigma \vert H_k \vert^p  \, \mathrm{d}\mu_k \\
&\geq \int_\Sigma \left( ( \xi_{k,2} H_k )^2 + \varepsilon \right)^{\frac{p}{2}} \, \mathrm{d}\mu_k + A^{1 - \frac{p}{2}} \left( \int_\Sigma \vert H_k \vert^2  \, \mathrm{d}\mu_k \right)^{\frac{p}{2}} \\
&\geq \int_\Sigma \left( ( \xi_{k,2} H_k )^2 + \varepsilon \right)^{\frac{p}{2}} \, \mathrm{d}\mu_k + A^{1 - \frac{p}{2}} \left( 4 \pi \right)^{\frac{p}{2}}
\end{align*}
where the inequality in the penultimate line comes the fixed-surface-area constraint and H\"{o}lder's inequality in the form
\[
\normm{H}_{L^2}^p \leq \normm{H}_{L^p}^p \normm{1}_{L^{2p/(p-2)}}^p,
\]
and the final line comes from the previously mentioned inequality of Willmore for any closed surface. \\
Taking $p^{\textnormal{th}}$ roots and using the assumption that $\mathcal{M}_p^{\sigma,\varepsilon}[f_k ; f_0]$ is bounded from above (as it is a minimising sequence) shows that translating the immersions reduces the value of $\mathcal{M}_p^{\sigma,\varepsilon}[f_k ; f_0]$ by at least some fixed constant. \\

\noindent The convergence of the sequence $(f_k)$ to some limit $f_\infty$ follows, and it remains to argue that $f_\infty$ is a minimiser of $\mathcal{M}_p^{\sigma,\varepsilon}[ \,\bigcdot\, ;f_0]$, i.e. we need to establish the inequality
\[
\liminf_k \mathcal{M}_p^{\sigma,\varepsilon}[ f_k ;f_0] \leq \inf_f \mathcal{M}_p^{\sigma,\varepsilon}[ \,f ;f_0]
\]
with the admissible set of immersions as given in Section \ref{subsec:Lpmin}. \\
First, the convergence of the penalisation term $\mathcal{P}^\sigma[f_\infty;f_0] = \lim_k \mathcal{P}^\sigma[f_k;f_0]$ holds because the weak $W^{2,p}$ convergence of the sequence $(f_k)$ implies strong $C^0$ convergence. \\
Second, Theorem \ref{thm:KLL} directly implies the inequality
\[
\left( \int_\Sigma \left( (\xi_\infty H_\infty)^2 + \varepsilon \right)^{\frac{p}{2}} \,\mathrm{d}\mu_\infty \right)^{\frac{1}{p}} \leq \liminf_k \left( \int_\Sigma \left( (\xi_k H_k)^2 + \varepsilon \right)^{\frac{p}{2}} \,\mathrm{d}\mu_k \right)^{\frac{1}{p}}.
\]
These two statements combined immediately show that $f_\infty$ is minimal for the $\mathcal{M}_p^{\sigma,\varepsilon}[ \,\bigcdot\, ;f_0]$ functional, which completes the proof of the Lemma.
\end{proof}

\noindent Using the tools provided in Lemma \ref{lem:ELcomps}, we can readily compute the Euler-Lagrange equations which critical points of $\mathcal{M}_p^{\sigma,\varepsilon}[ \,\bigcdot\, ;f_0]$ must satisfy:

\begin{prop}\label{prop:pELepseqs}
Let the immersion $f_{p,\varepsilon}$ be a critical point of $\mathcal{M}_p^{\sigma,\varepsilon}[ \,\bigcdot\, ;f_0]$ and take $\Delta_{p,\varepsilon}$, $H_{p,\varepsilon}$, $K_{p,\varepsilon}$, $\xi_{p,\varepsilon}$ and $d_{p,\varepsilon}$ to be defined analogously as in Proposition \ref{prop:pELeqs}. Set
\[
h_{p,\varepsilon} = A^{-\frac{1}{p}} \left( \int_\Sigma \left( (\xi_{p,\varepsilon} H_{p,\varepsilon})^2 + \varepsilon \right)^{\frac{p}{2}} \, \mathrm{d} \mu_{p,\varepsilon} \right)^{\frac{1}{p}},
\]
and define the functions
\[
w_{p,\varepsilon} = h_{p,\varepsilon}^{1-p} ((\xi_{p,\varepsilon} H_{p,\varepsilon})^2 + \varepsilon)^{\frac{p-2}{2}} \xi_{p,\varepsilon}^2 H_{p,\varepsilon}
\]
and
\[
Q_{p,\varepsilon} = 2H_{p,\varepsilon}^2 - K_{p,\varepsilon} - \frac{2}{p} \left( H_{p,\varepsilon}^2 + \frac{\varepsilon}{\xi_{p,\varepsilon}^2} \right) + \frac{H_{p,\varepsilon} D_{\nu_{p,\varepsilon}} \xi_{p,\varepsilon}}{\xi_{p,\varepsilon}}.
\]
Then the Euler-Lagrange equations for $f_{p,\varepsilon}$ are given by
\begin{equation}\label{eq:epsEL}
\frac{1}{2} \Delta_{p,\varepsilon} w_{p,\varepsilon} + Q_{p,\varepsilon} w_{p,\varepsilon} = \lambda_{p,\varepsilon} H_{p,\varepsilon} - \sigma d_{p,\varepsilon} D_{\nu_{p,\varepsilon}} d_{p,\varepsilon} + \sigma d_{p,\varepsilon}^2 H_{p,\varepsilon}
\end{equation}
where again the equation is to be interpreted in the weak sense, and the real number $\lambda_{p,\varepsilon}$ is the Lagrange multiplier arising to account for the area constraint $\mathcal{A}[f_{p,\varepsilon}] = A$.
\end{prop}

\noindent As was the case with Proposition \ref{prop:pELeqs}, Proposition \ref{prop:pELepseqs} follows from Lemma \ref{lem:ELcomps} via routine computations and so we do not provide a full proof. \\

\noindent We have the following result on the regularity of $w$:
\begin{prop}\label{prop:regularity}
Let $w_{p,\varepsilon}$ be as in Proposition \ref{prop:pELepseqs}, where $p > 2$ and $\varepsilon < \tilde{\varepsilon}$. Then $w_{p,\varepsilon} \in W^{2,2}(\Sigma)$.
\end{prop}
\noindent We omit a detailed proof of Proposition \ref{prop:regularity}, owing to the fact that the arguments used require only a very minor alteration of the arguments in \cite[Section 3.1]{kuwert2015two} used to prove the smoothness part of Theorem \ref{thm:KLL}. Because we are considering the weighted mean curvature $\xi H$ with $\xi$ as in Theorem \ref{thm:mainthm} (and in particular $\xi$ possibly non-smooth), we do not in general get full smoothness of $w$, although this is not a problem as our conclusion of the existence of the second order weak derivatives of $w$ is sufficient for our needs. \\

\noindent In fact the integrability of $w$ can be improved, although since the bounds provided in \cite{kuwert2015two} are not uniform this is not important to us; we only need to know that the second derivatives of $w$ exist in the weak sense so that our bounds in the following sections are valid. \\


\newpage
\section{Convergence of Approximations}\label{sec:compactness}

\subsection{Convergence of Immersions}\label{subsec:invconv}

\noindent We first show that our approximation of $\mathcal{M}_p^\sigma[\,\bigcdot\, ; f_0]$ by $\mathcal{M}_p^{\sigma,\varepsilon}[\,\bigcdot\, ; f_0]$ works on the level of immersions: that is, we can obtain a minimiser of $\mathcal{M}_p^\sigma[\,\bigcdot\, ; f_0]$ as the limit as $\varepsilon \rightarrow 0$ of a sequence of $\mathcal{M}_p^{\sigma,\varepsilon}[\,\bigcdot\, ; f_0]$ minimisers.
\begin{lem}\label{lem:epsconvimm}
For a given $\infty$-Willmore sphere satisfying the low-energy assumption of Theorem \ref{thm:mainthm} and a given fixed number $p \in (2,\infty)$, let $(f_k)$ be a sequence of minimisers of $\mathcal{M}_p^{\sigma,\varepsilon_k}[ \,\bigcdot\, ;f_0]$ with $\varepsilon_k \rightarrow 0$. Then there exists a minimiser $f_p$ of $\mathcal{M}_p^\sigma[\,\bigcdot\, ; f_0]$ such that $f_k \rightharpoonup f_p$ in the weak $W^{2,p}$ sense of Theorem \ref{thm:KLL}.
\end{lem}
\begin{proof}
The same computations as seen earlier in Section \ref{subsec:epsapprox} show that the assumptions of Theorem \ref{thm:KLL} hold for the sequence of immersions $(f_k)$, which therefore converges weakly in $W^{2,p}$ to some limiting immersion $f$. To show that $f =: f_p$ is indeed a minimiser of $\mathcal{M}_p^{\sigma}[ \,\bigcdot\, ;f_0]$, it suffices to consider the chain of inequalities for any given immersion $\tilde{f}$
\begin{align*}
\mathcal{M}_p^\sigma [ f ; f_0 ] &\leq \liminf_k \mathcal{M}_p^\sigma [ f_k ; f_0 ] \\
&\leq \liminf_k \mathcal{M}_p^{\sigma,\varepsilon_k} [ f_k ; f_0 ] \\
&\leq \liminf_k \mathcal{M}_p^{\sigma,\varepsilon_k} [ \tilde{f} ; f_0 ] \\
&= \mathcal{M}_p^{\sigma} [ \tilde{f} ; f_0 ]
\end{align*}
which is valid because of (respectively) the weak lower semicontinuity of $\mathcal{M}_p^\sigma[ \,\bigcdot\, ; f_0 ]$ with respect to $W^{2,p}$ convergence, the definitions of the $\mathcal{M}_p^\sigma[ \,\bigcdot\, ; f_0 ]$ and $\mathcal{M}_p^{\sigma,\varepsilon}[ \,\bigcdot\, ; f_0 ]$ functionals, the minimising characterisation of the immersions $(f_k)$, and the monotone convergence theorem. 
\end{proof}

\noindent Next, we show that our $L^p$ approximation works as a way of selecting $\infty$-Willmore spheres-- i.e. we can obtain any $\infty$-Willmore sphere $f_0$ satisfying the low energy assumption of Theorem \ref{thm:mainthm} as the limit of a sequence of $\mathcal{M}_p^{\sigma}[ \,\bigcdot\, ;f_0]$ minimisers:
\begin{prop}\label{prop:penconv}
Let $f_0$ be any given $\infty$-Willmore sphere which satisfies the low-energy assumption of Theorem \ref{thm:mainthm}, and let $(f_p)$ be a sequence of minimisers of $\mathcal{M}_p^{\sigma}[\,\bigcdot\, ; f_0]$. If $\sigma$ is chosen to be sufficiently large, the sequence $(f_p)$ converges to $f_0$ weakly in $W^{2,q}$ for all $q < \infty$.
\end{prop}

\begin{proof}
The subsequential convergence (again not made explicit) of the sequence $(f_p)$ in the weak $W^{2,q}$ sense for any finite $q$ follows from Theorem \ref{thm:KLL}. The calculations to ensure the necessary assumptions are satisfied follow along exactly the same lines as the calculations in Section \ref{subsec:epsapprox}. It remains to be shown that the limit, which we write as $f_\infty$ for now, is in fact $f_0$. \\
Suppose for a contradiction that $f_\infty$ differs from $f_0$. By the strong $C^{1,\beta}$ convergence of $(f_p)$ we have that $\lim_p \mathcal{P}^\sigma [ f_p ; f_0 ] = \mathcal{P}^\sigma [ f_\infty ; f_0 ]$, and because of the lower semicontinuity of the $L^q$ norm with respect to weak convergence we have that $\normm{\xi_\infty H_\infty}_{L^q} \leq \liminf_p \normm{\xi_p H_p}_{L^q}$. Combining these with the characterisation of the $L^\infty$ norm as the limit of $L^q$ norms as $q \rightarrow \infty$, we have that
\begin{align*}
\mathcal{M}_\infty^{\sigma}[f_\infty ; f_0 ] = \lim_q \mathcal{M}_q^{\sigma}[f_\infty ; f_0 ] \leq \lim_q \liminf_p \mathcal{M}_q^{\sigma}[f_p ; f_0 ].
\end{align*}
The chain of inequalities
\[
\mathcal{M}_q^{\sigma} [f_p ; f_0] \leq \mathcal{M}_p^{\sigma} [f_p ; f_0] \leq \mathcal{M}_p^{\sigma} [f_0 ; f_0] = \mathcal{M}_p [f_0] \leq \mathcal{M}_\infty [f_0]
\]
for $p \geq q$, which comes from H\"{o}lder's inequality, the definitions of the functionals involved and the definition of the sequence $(f_p)$, then implies that
\[
\mathcal{M}_\infty^{\sigma}[f_\infty ; f_0 ] \leq \mathcal{M}_\infty [f_0].
\]
By assumption $f_0$ is a $\infty$-Willmore sphere, and so there exists $M > 0$ such that
\begin{align*}
\mathcal{M}_\infty[f_0] &\leq \mathcal{M}_\infty[f_\infty] + \frac{M}{2A} \int_\Sigma \textnormal{dist}(f_\infty(x),f_0(\Sigma))^2 \, \mathrm{d}\mu_\infty \\
&\leq \mathcal{M}_\infty[f_\infty] + \frac{CM}{2A} \int_\Sigma d(f_\infty(x))^2 \, \mathrm{d}\mu_\infty
\end{align*}
where $d$ is some regularisation of the distance from $f_0(\Sigma)$ and the constant $C$ is as in \eqref{eq:distineqs}, so we have that
\[
\mathcal{M}_\infty^{\sigma}[f_\infty ; f_0] \leq \mathcal{M}_\infty^{CM}[f_\infty ; f_0].
\]
Picking $\sigma > CM$ forces the equality
\[
\int_\Sigma d(f_\infty(x),f_0(\Sigma))^2 \, \mathrm{d}\mu_0 = 0
\]
and so $f_0(\Sigma) \subseteq f_\infty(\Sigma)$. A straightforward topological argument then implies the equality of sets $f_0(\Sigma) = f_\infty(\Sigma)$. Both $f_0(\Sigma)$ and $f_\infty(\Sigma)$ are diffeomorphic to the round sphere $S$, so without loss of generality we may deform $f_\infty(\Sigma)$ into $S$. If $f_0(\Sigma) \neq f_\infty(\Sigma)$, the image of $f_0(\Sigma)$ under this deformation must be a strict subset of $S$, but it must also be diffeomorphic to $S$, which is impossible. \\
The above arguments show that any subsequence of $(f_p)$ has a further subsequence which converges to $f_0$, and hence the original sequence must itself converge to $f_0$.
\end{proof}

\noindent We have the following corollaries:
\begin{cor}\label{cor:distconv}
In the context of Proposition \ref{prop:penconv}, the sequence of functions $d_p(x) = d(f_p(x))$ converges uniformly to $0$ as $p \rightarrow \infty$.
\end{cor}

\begin{cor}\label{cor:lapbeltconv}
The coefficients of the Laplace-Beltrami operators $(\Delta_p)$ on $f_p(\Sigma)$ converge as $p \rightarrow \infty$ to the coefficients of the Laplace-Beltrami operator (in divergence form) on $f_0(\Sigma)$ in $C^{0,\beta}$ for all $\beta < 1$.
\end{cor}

\noindent These corollaries follows because the weak $W^{2,q}$ convergence of the immersions for all finite $q$ implies strong $C^{1,\beta}$ convergence for all $\beta < 1$ by Rellich-Kondrachov compactness. For the distance functions it suffices to have mere $C^0$ convergence, while the operator $\Delta_p$ is given in local co-ordinates by 
\[
\Delta_p = \frac{1}{\sqrt{\det g_p}} \partial_i \left( \sqrt{\det g_p} (g_p)^{ij} \partial_j  \right)
\]
and hence the metric components $(g_p)^{ij}$ (i.e. first derivatives of $f_p$) are convergent in $C^{0,\beta}$ to $(g_0)^{ij}$. \\

\noindent From this point onwards we will always assume with no loss of generality that $\sigma$ is large enough for Proposition \ref{prop:penconv} to apply. \\

\subsection{Uniform Bounds; Notation \& Terminology}\label{subsec:uniformbounds}

\noindent In this section we derive higher estimates on the terms in the Euler-Lagrange equation \eqref{eq:epsEL} which will be used to ensure that we have not just the convergence of immersions as seen in Section \ref{subsec:invconv} but also the convergence of the Euler-Lagrange equations \eqref{eq:pEL} and \eqref{eq:epsEL}. \\

\noindent Because we are interested in taking limits in two different ways-- first, the limit as $\varepsilon \rightarrow 0$ of our $\mathcal{M}_p^{\sigma,\varepsilon}[ \,\bigcdot\, ;f_0]$ minimisers and Euler-Lagrange equations for a fixed number $p$, and second, the limit as $p \rightarrow \infty$ of our $\mathcal{M}_p^{\sigma}[ \,\bigcdot\, ;f_0]$ minimisers and Euler-Lagrange equations-- we must be precise with our terminology. We have already established in Section \ref{subsec:epsapprox} that there exists some threshold $\tilde{\varepsilon}$ beneath which the functionals $\mathcal{M}_p^{\sigma,\varepsilon}[ \,\bigcdot\, ;f_0]$ have minimisers, and as $\varepsilon \rightarrow 0$ these minimisers converge to a minimiser of $\mathcal{M}_p^{\sigma}[ \,\bigcdot\, ;f_0]$. However, for some of our estimates on the terms of the Euler-Lagrange equations \eqref{eq:pEL} and \eqref{eq:epsEL} to be valid we will need to assume that our immersions are ``sufficiently close'' to the $\infty$-Willmore sphere $f_0$. This means that $p$ needs to be sufficiently large (so $f_p$ is close to $f_0$), say $p > \hat{p}$ for some $\hat{p}$, and $\varepsilon$ needs to be sufficiently small (so $f_{p,\varepsilon}$ is close to $f_p$). However, the rate of convergence of $f_{p,\varepsilon}$ to $f_p$ may well not be independent of $p$, and so we can only guarantee that our bounds hold when $p > \hat{p}$ and $\varepsilon$ is beneath some value $\hat{\varepsilon}(p)$. \\

\noindent Recall from Lemma \ref{lem:epsminexist} that there exists a constant $\tilde{\varepsilon}$ independent of $p$ such that for all $p > 2$ and $\varepsilon < \tilde{\varepsilon}$, the functional $\mathcal{M}_p^{\sigma,\varepsilon}[ \,\bigcdot\, ;f_0]$ admits a minimiser. Hence we can freely assume that $\hat{\varepsilon}(p) < \tilde{\varepsilon}$. \\

\noindent In view of the above discussions, when we make statements in this section of the form ``for all $p$ sufficiently large and $\varepsilon$ sufficiently small,'' it is to be understood that we mean it in the precise sense detailed above, for all $p > \hat{p}$ and $\varepsilon < \hat{\varepsilon}(p)$. When we talk about ``uniform bounds'' it will be meant unless explicitly stated otherwise that the bound in question does not depend on either $p$ or $\varepsilon$, although it may apply only for $p > \hat{p}$ sufficiently large and $\varepsilon < \hat{\varepsilon}(p)$ in a range depending on $p$. \\
Also, as we are allowing both $p$ and $\varepsilon$ to vary, for the convenience of our notation we will drop the subscripts from the immersions and elements of the Euler-Lagrange equation in Proposition \ref{prop:pELepseqs}, leaving it implicit that the immersions, mean curvatures, etc. that we talk about depend on $p$ and $\varepsilon$. \\
To demonstrate how we might use this terminology, the sentence ``$w$ is uniformly bounded in $L^1$ for all $p$ sufficiently large and $\varepsilon$ sufficiently small'' formally means that there exists a constant $C > 0$ such that for all $w = w_{p,\varepsilon}$ as in Proposition \ref{prop:pELepseqs} with $p > \hat{p}$ and $\varepsilon < \hat{\varepsilon}(p)$, the inequality $\normm{w}_{L^1} \leq C$ holds. \\

\noindent When writing uniform bounds, we will often use the notation $\lesssim$, which is commonly used in the PDE literature although not always explained. Formally, we say that $A \lesssim B$ for a set of pairs $(A_i,B_i)$, $i \in I$ some indexing set, if there exists a constant $C$ independent of $i$ such that $A_i \leq C B_i$ for all $i \in I$. In the context of this paper, the set $I$ will be given by the set of pairs $\{ (p, \varepsilon) \}$ with $p > \hat{p}$ and $\varepsilon < \hat{\varepsilon}(p)$. \\
For example, we may write the inequality
\[
\normm{H_{p,\varepsilon}}_{L^1} \leq  A^{\frac{1}{2}} \normm{H_{p,\varepsilon}}_{L^2}
\]
(which is valid by H\"{o}lder's inequality and the fixed-surface-area condition) as $\normm{H}_{L^1} \lesssim \normm{H}_{L^2}$, and the above uniform bound $\normm{w}_{L^1} \leq C$ becomes $\normm{w}_{L^1} \lesssim 1$.  \\

\noindent The purpose of this subsection is to prove the following statement:
\begin{prop}\label{prop:finalunifbounds}
Let $f_0$ be an $\infty$-Willmore surface satisfying the low energy assumption of Theorem \ref{thm:mainthm} and let $\sigma$ be large enough that Proposition \ref{prop:penconv} applies. Then for any fixed $p$ sufficiently large and all $\varepsilon$ sufficiently small, the following quantities are uniformly bounded in $\varepsilon$:
\[
\normm{f}_{W^{2,p}}, \, \normm{w}_{W^{2,\frac{p}{2}}}, \, \normm{Q}_{L^{\frac{p}{2}}}, \, \vert \lambda \vert, \, h.
\]
If we instead allow $p$ to vary, then for any fixed $q < \infty$ the following quantities are uniformly bounded in $p$ and $\varepsilon$ for all $p \geq 2q$ sufficiently large and $\varepsilon$ sufficiently small:
\[
\normm{f}_{W^{2,q}}, \, \normm{w}_{W^{2,q}}, \, \normm{Q}_{L^q}, \, \vert \lambda \vert, \, h.
\]
\end{prop}

\noindent We begin our path to proving Proposition \ref{prop:finalunifbounds} with the following lemma:
\begin{lem}\label{lem:simplebounds}
Let $q < \infty$ be given. For all $p \geq q$ sufficiently large and $\varepsilon$ sufficiently small, $\normm{f}_{W^{2,q}}$ is uniformly bounded. For all $p \geq 2q$ sufficiently large and $\varepsilon$ sufficiently small, the following quantities are uniformly bounded:
\[
\normm{H}_{L^p}, \normm{Q}_{L^q}, \normm{w}_{L^{p'}}, \normm{Hw}_{L^1}.
\]
\end{lem}

\noindent Here the quantity $p'$ denotes the arithmetic conjugate of $p$, given by
\[
p' = \frac{p}{p-1}
\]
so that
\[
\frac{1}{p} + \frac{1}{p'} = 1
\]
\noindent which is particularly useful in H\"{o}lder's inequality. \\
\begin{proof}[Proof of Lemma \ref{lem:simplebounds}]
The bound on $\normm{f}_{W^{2,q}}$ can be proved by contradiction: suppose instead that there exists a sequence $(p_k,\varepsilon_k)$ of immersions $f_k$ minimising $\mathcal{M}_{p_k}^{\sigma,\varepsilon_k}[\,\bigcdot\, ; f_0]$ minimisers with each $p_k > q$ and $\varepsilon_k < \hat{\varepsilon}(p_k)$ such that the real sequence $(\normm{f_k}_{W^{2,q}})$ tends to infinity. Without loss of generality, assume $\hat{\varepsilon}(p_k) \leq \tilde{\varepsilon}/2$ for all $k$. Using the same calculations as shown in Section \ref{subsec:epsapprox}, we see that $(f_k)$ satisfies the assumptions of Theorem \ref{thm:KLL} where in the statement of the theorem we take $p = q$ and the factor of $1/4$ is replaced by $\tilde{\varepsilon}/2$. The sequence $(f_k)$ must then admit a weakly convergent subsequence, hence a norm-bounded subsequence, which gives the contradiction. \\
The bound on $\normm{H}_{L^p}$ comes from the fact that $f$ is minimal for $\mathcal{M}_{p}^{\sigma,\varepsilon}[\,\bigcdot\, ; f_0]$. Indeed, by H\"{o}lder's inequality, the definitions of the functionals involved and the minimality assumption on $f$ we have that
\[
\normm{H}_{L^p} \leq A^{\frac{1}{p}} \mathcal{M}_{p}^{\sigma,\varepsilon}[ f ; f_0] \leq \textnormal{max}(1,A) \mathcal{M}_{p}^{\sigma,\varepsilon}[ \tilde{f} ; f_0] \leq \textnormal{max}(1,A) \mathcal{M}_{\infty}^{\sigma,\varepsilon}[ \tilde{f} ; f_0] 
\]
for some given immersion $\tilde{f}$. \\
To get the $\normm{Q}_{L^q}$ bound, we look at the separate components of $Q$, which we recall is given by
\[
Q = 2H^2 - K - \frac{2}{p} \left( H^2 + \frac{\varepsilon}{\xi^2} \right) + \frac{H D_{\nu} \xi}{\xi}.
\]
The $H^2$ terms are uniformly bounded in $L^q$ by H\"{o}lder's inequality, the fact that $p \geq 2q$ and the existing uniform bound on $\normm{H}_{L^p}$. The $\xi^{-1}$, $\xi^{-2}$ and $D_\nu \xi$ terms are bounded by conditions \eqref{cond:alph1}--\eqref{cond:alph2}. The $K$ term is uniformly $L^q$ bounded because $K$ depends on (quadratic terms in the) second derivatives of $f$ and these are controlled by the bounds on the immersion $f$ established earlier in the lemma. \\
Finally, the bounds on $\normm{w}_{L^{p'}}$ and $\normm{Hw}_{L^1}$ follow from the definition of $w$ and the $\normm{H}_{L^p}$ bound.
\end{proof}

\noindent It is also not too difficult to show the following bounds:
\begin{lem}\label{lem:Hqw2bounds}
For all $p \geq 16$ sufficiently large and $\varepsilon$ sufficiently small, the inequalities
\[
\int_\Sigma H^2 w^2 \, \mathrm{d}\mu, \int_\Sigma \vert H \vert w^2 \, \mathrm{d}\mu, \int_\Sigma w^2 \, \mathrm{d}\mu, \int_\Sigma \vert Q \vert w^2 \, \mathrm{d}\mu \lesssim \delta \normm{w^2}_{L^2} + \delta^{-3}
\]
hold for any $\delta > 0$.
\end{lem}

\begin{proof}[Proof of Lemma \ref{lem:Hqw2bounds}]
By H\"{o}lder's inequality, we have that
\begin{align*}
\int_\Sigma H^2 w^2 \, \mathrm{d}\mu &\leq \normm{H^2}_{L^8} \normm{w^{\frac{1}{2}}}_{L^2} \normm{w^{\frac{3}{2}}}_{L^{\frac{8}{3}}}. 
\intertext{The first two terms on the right-hand-side are uniformly bounded by Lemma \ref{lem:simplebounds} and H\"{o}lder's inequality. We therefore find that}
\int_\Sigma H^2 w^2 \, \mathrm{d}\mu &\lesssim \normm{w^{\frac{3}{2}}}_{L^{\frac{8}{3}}},
\intertext{and so Young's inequality implies that}
\int_\Sigma H^2 w^2 \, \mathrm{d}\mu &\lesssim \delta \normm{w^2}_{L^2} + \delta^{-3}.
\end{align*}
Essentially the same argument shows the claim for the remaining integrals.
\end{proof}

\noindent The compactness of $\Sigma$ allows us to make use of the Michael-Simon-Sobolev inequality \cite{Michael1973}, which we note here in the specific form we will apply it in: \\
\textbf{Michael-Simon-Sobolev Inequality}. \textit{The inequality
\[
\normm{w}_{L^2} \lesssim \left( \normm{\nabla w}_{L^1} + \normm{H w}_{L^1} \right)
\]
holds, where the norms and gradient are taken with respect to the immersion which induces $w$.} \\
The specific constant in the Michael-Simon-Sobolev inequality depends only on the dimension of the surface being immersed (in our case, two), not on the choice of function, nor the specific choice of immersion or even the topology of the surface. \\

\noindent We can now improve the uniform regularity of $w$ by obtaining a $W^{1,2}$ bound in terms of $\lambda$.
\begin{lem}\label{lem:w2bounds}
For all $p \geq 16$ sufficiently large and $\varepsilon$ sufficiently small, the inequality
\begin{equation}\label{eq:testreqs}
\normm{w}_{W^{1,2}}^2 \lesssim 1 + \vert \lambda \vert
\end{equation}
holds.
\end{lem}

\begin{proof}[Proof of Lemma \ref{lem:w2bounds}]
By the Michael-Simon-Sobolev inequality and H\"{o}lder's inequality we see that
\begin{align}
\normm{w^2}_{L^2} &\lesssim \normm{w}_{L^2} \normm{\nabla w}_{L^2} + \normm{H w^2}_{L^1} \nonumber \\
\intertext{and by Lemma \ref{lem:Hqw2bounds}, we can bound the last term to get}
\normm{w^2}_{L^2} &\lesssim \normm{w}_{L^2} \normm{\nabla w}_{L^2} + \delta\normm{w^2}_{L^2} + \delta^{-3} \nonumber  \\
\intertext{then choose $\delta$ to be small so that the inequality becomes}
\normm{w^2}_{L^2} &\lesssim \normm{w}_{L^2} \normm{\nabla w}_{L^2} + \delta^{-3}. \nonumber  \\
\intertext{Applying the Michael-Simon-Sobolev inequality again, we find that}
\normm{w^2}_{L^2} &\lesssim \left( \normm{\nabla w}_{L^1} + \normm{H w}_{L^1} \right) \normm{\nabla w}_{L^2} + \delta^{-3} \nonumber 
\intertext{then use Lemma \ref{lem:simplebounds} and H\"{o}lder's and Young's inequalities to write}
\normm{w^2}_{L^2} &\lesssim \normm{\nabla w}_{L^2}^2 + \delta^{-3}. \label{eq:messy1} \\
\intertext{Integrating by parts, this becomes}
\normm{w^2}_{L^2} &\lesssim \int_\Sigma \,\vert w \Delta w \vert \, \mathrm{d}\mu + \delta^{-3}. \nonumber 
\end{align}
We have that
\begin{align*}
w \Delta w &= - 2 Q w^2 + 2 \lambda H w - 2 \sigma w d D_\nu d + 2 \sigma d^2 w H
\end{align*}
from substituting in the Euler-Lagrange equation \eqref{eq:epsEL}, and so
\begin{align*}
\normm{\nabla w}_{L^2}^2 = \int_\Sigma w \Delta w \, \mathrm{d}\mu &\lesssim\int_\Sigma \vert Q \vert w^2 + \vert \lambda \vert w H + \sigma \vert w d D_\nu d \vert + \sigma \vert d^2 w H \vert \, \mathrm{d}\mu.
\end{align*}
The first term can be dealt with by Lemma \ref{lem:Hqw2bounds}, the second term is bounded by a constant multiple of $\lambda$ from Lemma \ref{lem:simplebounds}, and the third and fourth terms are bounded by Lemma \ref{lem:simplebounds} and the boundedness of $d$ and $D_\nu d$ (cf. \eqref{eq:distineqs}). Combining these, we have the inequality
\begin{equation}\label{eq:messy2}
\normm{\nabla w}_{L^2}^2 \lesssim \delta \normm{w^2}_{L^2} + \delta^{-3} + \vert \lambda \vert
\end{equation}
and this means inequality \eqref{eq:messy1} becomes
\[
\normm{w^2}_{L^2} \lesssim \delta^{-3} + \vert \lambda \vert
\]
which provides a uniform $L^2$ bound on $w$ in the form we desired after considering the inequality $\normm{w}_{L^2}^2 \lesssim \normm{w^2}_{L^2}$. Because we have a uniform $L^2$ bound on $w^2$ in terms of $\vert \lambda \vert$, inequality \eqref{eq:messy2} provides a uniform $L^2$ bound on $\nabla w$ of the same form, so we have the desired $W^{1,2}$ bound on $w$.
\end{proof}

\noindent To obtain a uniform $W^{1,2}$ bound on $w$, it remains only to bound $\vert \lambda \vert$. We first need the following result:
\begin{lem}\label{lem:testfnexist}
There exists a test function $\psi \in C^\infty_c(\Sigma)$ such that whenever $p$ is sufficiently large and $\varepsilon$ sufficiently small, then
\[
1 \leq \int_\Sigma \psi H \, \mathrm{d}\mu \leq 3.
\]
\end{lem}
\begin{proof}
For $\varepsilon > 0$ sufficiently small, the $W^{2,p}$ convergence $f_{p,\varepsilon} \rightharpoonup f_p$ as $\varepsilon \rightarrow 0$ implies that
\[
\int_\Sigma \psi H_{p} \, \mathrm{d}\mu_{p} - \delta \leq \int_\Sigma \psi H_{p,\varepsilon} \, \mathrm{d}\mu_{p,\varepsilon} \leq \int_\Sigma \psi H_{p} \, \mathrm{d}\mu_{g_{p}} + \delta
\]
for a given $\delta > 0$, where we have made the dependences on $p$ and $\varepsilon$ explicit. Then, taking $p$ sufficiently large, the $W^{2,q}$ convergence $f_p \rightharpoonup f_0$ implies similarly that
\[
\int_\Sigma \psi H_{0} \, \mathrm{d}\mu_{0} - \delta \leq \int_\Sigma \psi H_{p} \, \mathrm{d}\mu_{p} \leq \int_\Sigma \psi H_{0} \, \mathrm{d}\mu_{0} + \delta,
\]
so that
\[
\int_\Sigma \psi H_{0} \, \mathrm{d}\mu_{0} - 2\delta \leq \int_\Sigma \psi H_{p,\varepsilon} \, \mathrm{d}\mu_{p,\varepsilon} \leq \int_\Sigma \psi H_{0} \, \mathrm{d}\mu_{0} + 2\delta.
\]
Because $\Sigma$ is compact, we are implicitly assuming that no minimal surfaces are admissible (because of the well-known result that there are no compact minimal surfaces), i.e. we are assuming that it is not possible to have $H_0 \equiv 0$. Thus we may choose $\psi$ such that $\int_\Sigma \psi H_{0} \, \mathrm{d}\mu_{g_0} = 2$. Taking $\delta = 1/2$ then implies the result.
\end{proof}

\begin{lem}\label{lem:Lagrangebound}
For all $p$ sufficiently large and $\varepsilon$ sufficiently small, the uniform bound
\[
\vert \lambda \vert\, \leq C
\]
holds.
\end{lem}
\begin{proof}
Multiplying the Euler-Lagrange equation \eqref{eq:epsEL} by the test function $\psi$ from Lemma \ref{lem:testfnexist} and integrating by parts, we see that for appropriate $p, \varepsilon$ we have the inequality
\begin{align*}
\vert \lambda \vert \, &\leq \vert \lambda \vert \int_\Sigma \psi H \, \mathrm{d} \mu \\
&= \pm \int_\Sigma \frac{1}{2} w \Delta \psi + \psi Q w + \sigma d D_\nu d - \sigma d^2 H \, \mathrm{d} \mu \\
&\leq \int_\Sigma \vert w \Delta \psi \vert \, \mathrm{d}\mu + \int_\Sigma \vert \psi Q w \vert \, \mathrm{d} \mu + \sigma \int_\Sigma d \left( \vert D_\nu d \vert + d \vert H \vert \right) \, \mathrm{d}\mu \\
&:= I_1 + I_2 + I_3.
\end{align*}
By Young's inequality and the area bounds we find that
\begin{align*}
I_1 &\leq \delta \normm{w^2}_{L^1} + \delta^{-1} \int_\Sigma \vert \Delta \psi \vert^2 \, \mathrm{d}\mu \\
&\lesssim \delta \normm{w^2}_{L^2} + \delta^{-1} \\
&\lesssim \delta\vert \lambda \vert\, +\, \delta + \delta^{-1},
\end{align*}
where the bound on the integral of $\vert \Delta \psi \vert^2$ comes from the $W^{2,2}$ bound on the immersions $f_{p,\varepsilon}$ and we have used Lemma \ref{lem:w2bounds}. \\
We also find using Lemmas \ref{lem:simplebounds} and \ref{lem:Hqw2bounds} that
\begin{align*}
I_2 &\leq \int_\Sigma \vert Q \vert \, \mathrm{d}\mu + \int_\Sigma \psi^2 \vert Q \vert w^2 \, \mathrm{d}\mu \\
&\lesssim 1 + \delta \normm{w^2}_{L^2} + \delta^{-3} \\
&\lesssim 1 + \delta \vert \lambda \vert\, +\, \delta + \delta^{-3}.
\end{align*}
Finally, the uniform $L^\infty$ bounds on $d$ and $\vert D d \vert$ and the uniform $L^p$ bound on $H$ allows us to control $I_3$ by a constant: $I_3 \lesssim 1$. \\
We may now combine the inequalities for the three above integrals and obtain
\begin{align*}
\vert \lambda \vert\, &\lesssim 1 + \delta^{-3} + \delta \vert \lambda \vert
\end{align*}
from which we get the desired result by taking $\delta$ sufficiently small.
\end{proof}

\noindent By the Sobolev embedding theorem, the uniform $W^{1,2}$ bounds on $w$ imply uniform $L^q$ bounds on $w$ for any given $q < \infty$ (this also requires us to make use of the $L^\infty$ bounds we have on the metric components to control the constant in the embedding theorem). \\

\noindent To obtain the bounds we need on the second derivatives of $w$, we will require the following results: \\

\noindent \textbf{Bochner's Formula}. \textit{For any function $u \in W^{2,2}(\Sigma)$,}
\[
\int_\Sigma \left( \Delta u \right)^2 \, \mathrm{d}\mu = \int_\Sigma \left\vert \nabla^2 u \right\vert^2 \,\mathrm{d}\mu + \int_\Sigma K \left\vert \nabla u \right\vert^2 \, \mathrm{d}\mu.
\]
\noindent This formula can be found in, for example, the book \cite{lee2018introduction}. 

\begin{lem}\label{lem:D2ineq}
There exists a constant $C > 0$ such that for any function $u \in W^{2,2}(\Sigma)$ the inequality
\[
\left\vert \nabla \left( \left\vert \nabla u \right\vert^2 \right) \right\vert \leq C \left\vert \nabla u \right\vert \left\vert \nabla^2 u \right\vert
\]
is satisfied a.e. on $\Sigma$.
\end{lem}
\begin{proof}[Proof of Lemma \ref{lem:D2ineq}]
Taking normal co-ordinates around a point $x_0 \in M$, the inequality can be seen to hold at $x_0$ from straightforward computations. Since $x_0$ was arbitrary, the lemma follows.
\end{proof}

\noindent We now possess the tools to show the following bound:

\begin{lem}\label{lem:W2qbounds}
Let $q < \infty$ be given. Then the quantity
\[
\int_\Sigma \vert \nabla^2 w \vert^q \, \mathrm{d}\mu
\]
is uniformly bounded for all $p \geq 2q$ sufficiently large and $\varepsilon$ sufficiently small.
\end{lem}
\begin{proof}
We first obtain bounds in the special case $q = 2$. We cannot apply Bochner's formula immediately because we only have an $L^2$ bound on the gradient $\nabla w$ and an $L^q$ bound on the Gauss curvature $K$ (from the bounds on the immersions, cf. Lemma \ref{lem:simplebounds}), so we will first improve the integrability of $\nabla w$ and then use Bochner's formula. To this end let $\delta \in (0,2)$ be given and consider
\begin{align*}
\int_\Sigma \left\vert \nabla w \right\vert^{2 + \delta} \, \mathrm{d}\mu &= \int_\Sigma \left\langle \nabla w, \left\vert \nabla w \right\vert^{\delta} \nabla w \right\rangle \, \mathrm{d}\mu \\
&= - \int_\Sigma w \, \textnormal{div}\left( \left\vert \nabla w \right\vert^{\delta} \nabla w \right) \, \mathrm{d}\mu \\
&= - \int_\Sigma w \left\vert \nabla w \right\vert^{\delta} \Delta w + w \left\langle \nabla \left( \left\vert \nabla w \right\vert^{\delta} \right), \nabla w \right\rangle \, \mathrm{d}\mu \\
&= - \int_\Sigma w \left\vert \nabla w \right\vert^{\delta} \Delta w \, \mathrm{d}\mu - \int_\Sigma w \left\langle \nabla \left( \left( \left\vert \nabla w \right\vert^{2} \right)^{\delta/2} \right), \nabla w \right\rangle \, \mathrm{d}\mu \\
&\leq \int_\Sigma \vert w \vert \left\vert \nabla w \right\vert^{\delta} \vert \Delta w \vert \, \mathrm{d}\mu + \frac{\delta}{2} \int_\Sigma \vert w \vert \left\vert \nabla w \right\vert^{\delta - 1} \left\vert \nabla \left( \left\vert \nabla w \right\vert^{2} \right) \right\vert \, \mathrm{d}\mu \\
&\leq \int_\Sigma \vert w \vert \left\vert \nabla w \right\vert^{\delta} \vert \Delta w \vert \, \mathrm{d}\mu + \frac{C \delta}{2} \int_\Sigma \vert w \vert \left\vert \nabla w \right\vert^{\delta} \left\vert \nabla^2 w \right\vert \, \mathrm{d}\mu \\
\intertext{(using Lemma \ref{lem:D2ineq})}
&= I_1 + I_2.
\end{align*}
First we deal with the $I_1$ term: by applying Young's inequality with three terms to
\[
\vert w \vert \left\vert \nabla w \right\vert^{\delta} \left\vert \Delta w \right\vert = \left( \epsilon^{-\frac{3\delta + 2}{2(2 + \delta)}} \vert w \vert \right) \left( \epsilon^{\frac{\delta}{2 + \delta}} \left\vert \nabla w \right\vert^\delta \right) \left( \epsilon^{\frac{1}{2}} \left\vert \Delta w \right\vert \right)
\]
we can write
\[
I_1 \leq \frac{(2 - \delta)\epsilon^{-\frac{3\delta + 2}{2 - \delta}}}{2(2 + \delta)} \int_\Sigma \vert w \vert^{\frac{2(2 + \delta)}{2 - \delta}} \, \mathrm{d}\mu + \frac{\delta \epsilon}{2 + \delta} \int_\Sigma \left\vert \nabla w \right\vert^{2 + \delta} \, \mathrm{d}\mu + \frac{\epsilon}{2} \int_\Sigma \left\vert \Delta w \right\vert^2 \, \mathrm{d}\mu.
\]
The first and third terms are uniformly bounded in $p$ because of Lemmas \ref{lem:simplebounds}, \ref{lem:w2bounds}, and \ref{lem:Lagrangebound}, and the Euler-Lagrange equation \eqref{eq:epsEL}. We can ignore the second term if we pick $\epsilon$ small enough (by Proposition \ref{prop:regularity} the integral must be finite). \\
\noindent We now deal with the $I_2$ term. We use essentially the same trick with Young's inequality, but the calculations are slightly more involved due to the appearance of the Hessian $\nabla^2 w$ instead of the Laplacian $\Delta w$, which requires us to use Bochner's formula. Specifically, we apply Young's inequality to
\[
\vert w \vert \left\vert \nabla w \right\vert^{\delta} \left\vert \nabla^2 w \right\vert = \left( \epsilon^{-\frac{3\delta + 2}{2(2 + \delta)}} \vert w \vert \right) \left( \epsilon^{\frac{\delta}{2 + \delta}} \left\vert \nabla w \right\vert^\delta \right) \left( \epsilon^{\frac{1}{2}} \left\vert \nabla^2 w \right\vert \right)
\]
and find that
\begin{align*}
I_2 &\leq \frac{(2 - \delta)\epsilon^{-\frac{3\delta + 2}{2 - \delta}}}{2(2 + \delta)} \int_\Sigma \vert w \vert^{\frac{2(2 + \delta)}{2 - \delta}} \, \mathrm{d}\mu + \frac{\delta \epsilon}{2 + \delta} \int_\Sigma \left\vert \nabla w \right\vert^{2 + \delta} \, \mathrm{d}\mu + \frac{\epsilon}{2} \int_\Sigma \left\vert \nabla^2 w \right\vert^2 \, \mathrm{d}\mu.
\end{align*}
The first term is bounded uniformly in $p$ by our existing bounds on $w$ (Lemmas \ref{lem:w2bounds} and \ref{lem:Lagrangebound}), the second term can be ignored by picking $\epsilon$ sufficiently small, and the third term can be bounded using Bochner's formula, Young's inequality, and the uniform $L^q$ bounds for $K$ by
\begin{align*}
\frac{\epsilon}{2} \int_\Sigma \left\vert \nabla^2 w \right\vert^2 \, \mathrm{d}\mu &= \frac{\epsilon}{2} \int_\Sigma \left\vert \Delta w \right\vert^2 \, \mathrm{d}\mu - \frac{\epsilon}{2} \int_\Sigma K \left\vert \nabla w \right\vert^2 \, \mathrm{d}\mu \\
&\leq \frac{\epsilon}{2} \int_\Sigma \left\vert \Delta w \right\vert^2 \, \mathrm{d}\mu + \frac{\epsilon}{2 + \delta} \int_\Sigma \left\vert \nabla w \right\vert^{2 + \delta} \, \mathrm{d}\mu + \frac{\delta \epsilon}{2(2 + \delta)} \int_\Sigma \vert K \vert^{\frac{2+\delta}{\delta}} \, \mathrm{d}\mu.
\end{align*}
Again, the first and third integrals are uniformly bounded in $p$ by the Euler-Lagrange equation \eqref{eq:epsEL} and our existing bounds, and the second integral can be moved to the left-hand-side for small $\epsilon$. \\
\noindent This argument gives us a $W^{1,2 + \delta}$ bound on $w$ for chosen $\delta < 2$, and this is the last ingredient we need to go back to Bochner's formula which then gives the $W^{2,2}$ bound. \\
\noindent To improve the integrability of $\nabla^2 w$ to $L^q$ for given $q < \infty$, we employ standard elliptic regularity theory. The function $w$ solves the Dirichlet problem $L w = \zeta$ in $\Sigma$ for $L = \frac{1}{2}\Delta$ given as half the Laplace-Beltrami operator and
\[
\zeta = \lambda H - Q w - \sigma d D_\nu d + \sigma d^2 H \in L^{\frac{p}{2}}(\Sigma),
\]
where the $L^{\frac{p}{2}}$ regularity comes from our prior regularity results on all the terms on the right-hand-side. Here the $L^{\frac{p}{2}}$ regularity is the best we can attain, because of the presence of the Gauss curvature in the $Q$ term. We do not have a uniform $L^{\frac{p}{2}}$ bound on $\zeta$, but from the form of $\zeta$ we can see that so long as $q < \frac{p}{2}$, $\zeta$ is uniformly bounded in $L^q$ (hence by taking $p$ large we get a uniform $L^q$ bound for any fixed $1 \leq q < \infty$). By standard existence/uniqueness results for the Dirichlet problem \cite{gilbarg1977elliptic} and the compactness of $\Sigma$, after recalling Corollary \ref{cor:lapbeltconv}, we obtain the same $L^q$ uniform bounds on $\nabla^2 w$.
\end{proof}

\begin{proof}[Proof of Proposition \ref{prop:finalunifbounds}]
\noindent All of Proposition \ref{prop:finalunifbounds}, apart from the statements involving the scalars $h$, follows from Lemmas \ref{lem:simplebounds}, \ref{lem:Lagrangebound}, and \ref{lem:W2qbounds}. The remainder is a straightforward consequence of the definition of $h$ in Proposition \ref{prop:pELepseqs}.
\end{proof}

\subsection{Convergence of $w$ as $\varepsilon \rightarrow 0$ for fixed $p$} \label{subsec:epsconv}

\noindent In this subsection, we prove the following result:
\begin{prop}\label{prop:epszerobounds}
Let the given number $p$ be sufficiently large. There exists a minimiser $f$ of the functional $\mathcal{M}_p^\sigma[\,\bigcdot\, ; f_0]$ such that the corresponding function $w$ defined in Proposition \ref{prop:pELeqs} possesses the following regularity:
\[
w \in W^{2,\frac{p}{2}}(\Sigma).
\]
In addition, $f$, $w$, and all the terms involved in the Euler-Lagrange equation \eqref{eq:pEL} satisfy the fixed-$p$ uniform bounds of Proposition \ref{prop:finalunifbounds}. 
\end{prop}

\noindent We have already shown in Lemma \ref{lem:epsconvimm} that for a sequence of immersions $(f_n)$ which minimise $\mathcal{M}_p^{\sigma,\varepsilon_n}[ \,\bigcdot\, ;f_0]$, where $(\varepsilon_n)$ is a sequence tending to zero, we have the convergence of $f_n$ to some minimiser $f$ of $\mathcal{M}_p^{\sigma}[ \,\bigcdot\, ;f_0]$. We will show now that the functions $w_n$ converge to the function $w$ corresponding to $f$ via Proposition \ref{prop:pELeqs}. \\

\noindent We first need to establish a convergence result we do not have already:
\begin{lem}\label{lem:epsconvh}
Let $(f_k)$ be a sequence of minimisers of $\mathcal{M}_p^{\sigma,\varepsilon_k}[ \,\bigcdot\, ;f_0]$ where $\varepsilon_k \rightarrow 0$, converging by Lemma \ref{lem:epsconvimm} to a limit $f$ minimising $\mathcal{M}_p^\sigma[\,\bigcdot\, ; f_0]$. Then the scalars $(h_k)$ defined in Proposition \ref{prop:pELepseqs} with respect to $(f_k)$ converge to the scalar $h$ defined in Proposition \ref{prop:pELeqs} with respect to $f$.
\end{lem}
\begin{proof}
On the one hand, for each $k \in \mathbb{N}$ we have by the definitions of $h$ and $h_k$ (cf. Propositions \ref{prop:pELeqs} and \ref{prop:pELepseqs}) and the minimality of $f$ that
\begin{align*}
h + \mathcal{P}^\sigma[f;f_0] &= \mathcal{M}_p^\sigma [f ; f_0] \\
&\leq \mathcal{M}_p^\sigma [f_k ; f_0] \\
&\leq \mathcal{M}_p^{\sigma,\varepsilon_k} [f_k ; f_0] \\
&= h_k + \mathcal{P}^\sigma[f_k ; f_0],
\end{align*}
As in the proof of Lemma \ref{lem:epsminexist}, the weak $W^{2,p}$ convergence $f_k \rightharpoonup f$ implies the convergence of the penalisation functionals $\mathcal{P}^\sigma[f_k ; f_0] \rightarrow \mathcal{P}^\sigma[f;f_0]$. Therefore, we can take the $\liminf$ on both sides of the above chain of inequalities and deduce that $h \leq \liminf_k h_k$. On the other hand, because $f_k$ is minimal for $\mathcal{M}_p^{\sigma,\varepsilon_k}[ \,\bigcdot\, ;f_0]$ we may write that
\[
h_k \leq \left( \int_\Sigma ( \vert \xi H \vert^2 + \varepsilon_k )^{\frac{p}{2}} \, \mathrm{d}\mu \right)^{\frac{1}{p}}
\]
and so by the monotone convergence theorem and the definition of $h$ we find that $\limsup_k h_k \leq h$.
\end{proof}

\begin{proof}[Proof of Proposition \ref{prop:epszerobounds}]
\noindent As before, let $(f_k)$ be a sequence of minimisers of the functionals $\mathcal{M}_p^{\sigma,\varepsilon_k}[ \,\bigcdot\, ;f_0]$ where $\varepsilon_k \rightarrow 0$. By Proposition \ref{prop:finalunifbounds} we know that the sequence $(w_k)$ corresponding to the sequence $(f_k)$ via the Euler-Lagrange equations \eqref{eq:epsEL} is bounded in $W^{2,p/2}$, thus it admits a subsequence (not explicitly labelled) which converges weakly in the $W^{2,p/2}$ sense. That is, labelling the limit as $\tilde{w}$, we have that
\begin{equation}\label{eq:wepsconv}
h_k^{1-p} ((\xi_k H_k)^2 + \varepsilon_k)^{\frac{p-2}{2}} \xi_k^2 H_k = w_k \rightharpoonup \tilde{w} \in W^{2,\frac{p}{2}}(\Sigma)
\end{equation}
with $h_k$, $\xi_k$, $\varepsilon_k$ and $H_k$ as in Proposition \ref{prop:pELepseqs}. We will show that $\tilde{w}$ is in fact given by the function $w$ as in the statement of Proposition \ref{prop:epszerobounds}. We have the following modes of convergence:
\begin{equation*}
\begin{dcases}
\textnormal{$w_k \rightarrow \tilde{w}$ uniformly by the weak $W^{2,p/2}$ convergence $w_k \rightharpoonup \tilde{w}$,} \\
\textnormal{$h_k \rightarrow h$ in $\mathbb{R}$ by Lemma \ref{lem:epsconvh},} \\
\textnormal{$\xi_k \rightarrow \xi$ uniformly by the uniform convergence $f_k \rightarrow f$,} \\
\textnormal{$\varepsilon_k \rightarrow 0$ in $\mathbb{R}$ by assumption,}
\end{dcases}
\end{equation*}
and so from equation \eqref{eq:wepsconv} and the fact that the function
\[
t \mapsto h_k^{1-p} ((\xi_k t)^2 + \varepsilon_k)^{\frac{p-2}{2}} \xi_k^2 t
\]
is invertible, we see that we must have the uniform convergence of $H_k$ to some limit. However, we already know that $H_k \rightharpoonup H$ weakly in $L^p$ where $H$ is the mean curvature of $f$ because of Remark \ref{rem:meancurvconv}. By the coincidence of the weak and uniform limits we see that $H_k \rightarrow H$ uniformly. Going back to \eqref{eq:wepsconv}, it follows that
\[
\tilde{w} = h^{1-p} ((\xi H)^2 + \varepsilon)^{\frac{p-2}{2}} \xi^2 H = w
\]
with $w$ as in Proposition \ref{prop:pELeqs}, i.e. we have that $w \in W^{2,p/2}(\Sigma)$.
\end{proof}

\begin{rem}\label{rem:epszerobounds}
Because we now know that $w \in W^{2,p/2}(\Sigma)$, it is possible to apply all the arguments used in Section \ref{subsec:uniformbounds} to the Euler-Lagrange equation \eqref{eq:pEL} as well, so all the uniform bounds of Proposition \ref{prop:finalunifbounds} also apply in the limiting $\varepsilon = 0$ case.
\end{rem}

\subsection{Convergence of Euler-Lagrange as $p \rightarrow \infty$} \label{subsec:pconv}

\noindent In the previous subsection we saw that for any fixed $p$ there exists a minimiser $f$ of the functional $\mathcal{M}_p^{\sigma}[ \,\bigcdot\, ;f_0]$ whose corresponding function $w$ as in Proposition \ref{prop:pELeqs} is $W^{2,p/2}$ regular, and that the terms in the Euler-Lagrange equation \eqref{eq:pEL} obey the uniform bounds of Proposition \ref{prop:finalunifbounds} (cf. Remark \ref{rem:epszerobounds}). In this section we will begin the proof of Theorem \ref{thm:mainthm}: we will obtain $\infty$-Willmore spheres and their limiting equation system \eqref{eq:main1}--\eqref{eq:main2} by taking the limit as $p \rightarrow \infty$ of both a sequence of minimisers of $\mathcal{M}_p^{\sigma}[ \,\bigcdot\, ;f_0]$ as constructed in Section \ref{subsec:epsconv} and of their Euler-Lagrange equations \eqref{eq:pEL}. We must therefore adjust our notation to keep track of the index $p$ as it varies, and so we write $f_p$ for the minimiser of $\mathcal{M}_p^{\sigma}[ \,\bigcdot\, ;f_0]$ and additionally include the subscript $p$ in all the terms in the Euler-Lagrange equations \eqref{eq:pEL}. \\

\noindent We require a corollary of Proposition \ref{prop:penconv}: namely, that we gain control over the convergence of the sequence $(h_p)$ defined in Proposition \ref{prop:pELeqs}:
\begin{cor}\label{cor:hpconv}
In the context of Proposition \ref{prop:penconv}, the sequence $(h_p)$ converges to the limit given by
\[
h = \normm{\xi_0 H_0}_{L^\infty}.
\]
\end{cor}
\begin{proof}
On the one hand,
\begin{align*}
\normm{\xi_0 H_0}_{L^\infty} &= \lim_q \normm{\xi_0 H_0}_{L^q} \\
&\leq \lim_q \liminf_p \normm{\xi_p H_p}_{L^p} A^{\frac{1}{q} - \frac{1}{p}} \\
&= \lim_q \liminf_p h_p A^{\frac{1}{q}} \\
&= \liminf_p h_p
\end{align*}
using the weak lower semicontinuity of the $L^q$ norm, H\"{o}lder's inequality, and the prescribed surface area. \\
On the other hand,
\begin{align*}
\limsup_p h_p &= \limsup_p A^{-\frac{1}{p}} \normm{\xi_p H_p}_{L^p} \\
&\leq \limsup_p \mathcal{M}_p^\sigma [ f_p ; f_0] \\
&\leq \limsup_p \mathcal{M}_p^\sigma [ f_0 ; f_0] \\
&\leq \limsup_p \mathcal{M}_\infty^\sigma [ f_0 ; f_0] \\
&= \mathcal{M}_\infty^\sigma [ f_0 ; f_0] \\
&= \normm{\xi_0 H_0}_{L^\infty}
\end{align*}
using the characterisation of $(f_p)$ as a sequence of minimisers of $\mathcal{M}_p^\sigma [ \,\bigcdot\, ; f_0]$, the definition of $\mathcal{M}_p^\sigma [ \,\bigcdot\, ; f_0]$, and H\"{o}lder's inequality.
\end{proof}

\noindent We now show that the limiting system of equations \eqref{eq:main1}--\eqref{eq:main2} is satisfied on $f_0$. Consider the Euler-Lagrange equations \eqref{eq:pEL},
\[
\frac{1}{2} \Delta_p w_p + Q_p w_p = \lambda_p H_p - \sigma d_p D_{\nu_p} d_p + \sigma d_p^2 H_p,
\]
repeated here for convenience. By Proposition \ref{prop:finalunifbounds}, Remark \ref{rem:meancurvconv} and Corollary \ref{cor:distconv}, we have the existence of functions $w$ and $Q$ such that the following modes of convergence hold (modulo taking subsequences): \\
\begin{equation*}
\begin{dcases}
w_p \rightharpoonup w \text{ in $W^{2,q}$ for all $q < \infty$}, \\
Q_p \rightharpoonup Q \text{ in $L^{q}$ for all $q < \infty$}, \\
\lambda_p \rightarrow \lambda \text{ in $\mathbb{R}$}, \\
H_p \rightharpoonup H \text{ in $L^{q}$ for all $q < \infty$}, \\
d_p \rightarrow 0 \text{ in $C^0$}, \\
\end{dcases}
\end{equation*}
Also, by Corollary \ref{cor:lapbeltconv} the coefficients of $\Delta_p$ converge in $C^0$ to the coefficients of the Laplace-Beltrami operator on $f(\Sigma)$, and by construction of the regularised distance $d$ the bound $\normm{D_{\nu_p} d_p}_{L^\infty} \leq C$ holds. Taking the limit of equation \eqref{eq:pEL}, the first limiting equation \eqref{eq:main1} follows. \\
The second limiting equation \eqref{eq:main2} is trivially satisfied at zeroes of $w$, and to see that it holds where $w \neq 0$ we consider the equation
\begin{equation}\label{eq:Hpconv}
H_p = \vert w_p \vert^{\frac{1}{p-1}} h_p \frac{w_p}{\vert w_p \vert} \xi^{-\frac{p}{p-1}},
\end{equation}
which follows from the definition of $w_p$. Away from zeroes of $w$, we have the following modes of convergence due to Corollary \ref{cor:hpconv} and the uniform convergence of $w_p \rightarrow w$ and $f_p \rightarrow f$:
\begin{equation*}
\begin{dcases}
\vert w_p \vert^{\frac{1}{p-1}} \rightarrow 1 \text{ locally uniformly}, \\
h_p \rightarrow h \text{ in $\mathbb{R}$}, \\
\frac{w_p}{\vert w_p \vert} \rightarrow \frac{w}{\vert w \vert} \text{ locally uniformly} \\
\xi_p \rightarrow \xi \text{ locally uniformly},
\end{dcases}
\end{equation*}
which shows that the right-hand-side of equation \eqref{eq:Hpconv} converges locally uniformly away from zeroes of $w$. Hence the left-hand-side of \eqref{eq:Hpconv} converges in the same way, and because we already know that $H_p \rightharpoonup H$ in $L^q$ for finite $q$ then the limits must coincide and we in fact have the local uniform convergence of $H_p$ to $H$ away from zeroes of $w$, so equation \eqref{eq:main2} holds almost everywhere on $\Sigma$. \\

\noindent We remark that the limiting $Q$ in \eqref{eq:main1} is not necessarily expressible in the same form as the $Q_p$ terms, because the square of the mean curvature $H_p^2$ and the Gauss curvature $K_p$ involve terms which are quadratic in the second derivatives of $f$, and in general the product of weakly convergent sequences does not converge weakly to the product of the respective limits.


\newpage
\section{Behaviour and Existence of $\infty$-Willmore spheres}\label{sec:limitsystem}

\noindent Having established that the limiting equations \eqref{eq:main1}--\eqref{eq:main2} are satisfied on any given $\infty$-Willmore sphere $f_0$, which we now relabel as $f$ to fit into the notation of Theorem \ref{thm:mainthm}, we can now prove the rest of Theorem \ref{thm:mainthm}: that the function $w$ does not vanish everywhere, and that it governs the behaviour of the mean curvature of $f$.
\begin{prop}\label{prop:nonvanishing}
In the context of Theorem \ref{thm:mainthm}, the function $w$ is not identically zero.
\end{prop}

\begin{proof}[Proof of Proposition \ref{prop:nonvanishing}]
This statement is an immediate consequence of the lower bound away from zero on $\normm{w_p}_{L^{p'}(\Sigma)}$ from the definition of $w_p$ and the uniform convergence $w_p \rightarrow w$.
\end{proof}

\noindent Now that we know that $w \not\equiv 0$, we can analyse the system \eqref{eq:main1}--\eqref{eq:main2} to prove the remainder of Theorem \ref{thm:mainthm}.

\begin{prop}\label{prop:threevalues}
In the context of Theorem \ref{thm:mainthm}, the function $\xi(f(x)) H(x)$ takes on only three values up to null sets: $\xi H = h \, \sgn(w)$ where $w \neq 0$ and $\xi H = 0$ where $w = 0$.
\end{prop}
\begin{proof}
At points $x$ where $w(x) \neq 0$, the result follows immediately from equation \eqref{eq:main2}. Elsewhere, we consider two different cases depending on the value of $\lambda$. First, if $\lambda \neq 0$, then at any given $x \in \Sigma$ where $w(x) = 0$ and $H(x) \neq 0$ (recall that $\xi$ cannot be zero as we have assumed $\xi \geq 1$), equation \eqref{eq:main1} implies that $\Delta w \neq 0$. Because $w$ is in $W^{2,q}(\Sigma)$ with $q > 2$, its first derivatives are differentiable almost everywhere (this result can be found, for example, in \cite[Theorem 6.5]{evans2018measure}). Given any subset $\Omega \subset \Sigma$ diffeomorphic to the unit ball $B_1(0)$, take local co-ordinates $(u,v)$ and consider the foliation of $\Omega$ by a family of lines $l_{x,\theta} (t) = (u_1 + t\cos(\theta), v_1 + t\sin(\theta))$ centred at $x = f(u_1,v_1)$ where the familial parameter $\theta$ ranges over $[0,\pi)$. We claim that for almost every $x \in \Omega$, the function $w_{x,\theta}(t) := w(l_{x,\theta}(t))$ is twice differentiable at $0$ for almost every $\theta \in [0,\pi)$. Indeed, assume that there exists a non-null subset $\Omega' \subset \Omega$ such that for every $x \in \Omega'$, there is a non-null subset $A_x \subset [0,\pi)$ with the property that $w_{x,\theta}$ is not twice differentiable at $0$ for every $\theta \in A_x$. Introducing the integral
\begin{align*}
I &= \int_0^\pi \left\vert x \in \Omega \,:\, w_{x,\theta}''(0) \textnormal{ exists} \right\vert \, \mathrm{d}\theta, \\ 
\intertext{it follows by Fubini's theorem and standard computations that}
I &= \int_\Omega \left\vert \theta \in [0,\pi) \,:\, w_{x,\theta}''(0) \textnormal{ exists} \right\vert \, \mathrm{d}x \\
&= \int_{\Omega'} \left\vert \theta \in [0,\pi) \,:\, w_{x,\theta}''(0) \textnormal{ exists} \right\vert \, \mathrm{d}x + \int_{\Omega \backslash \Omega'} \left\vert \theta \in [0,\pi) \,:\, w_{x,\theta}''(0) \textnormal{ exists} \right\vert \, \mathrm{d}x  \\
&\leq \int_{\Omega'} \left\vert [0,\pi) \backslash A_x \right\vert \, \mathrm{d}x + \int_{\Omega \backslash \Omega'} \pi \, \mathrm{d}x  \\
&< \pi \left( \vert \Omega' \vert + \vert \Omega \backslash \Omega' \vert \right),
\end{align*}
where the final inequality comes from the assumption on $A_x$, and hence we have that $I < \pi \vert \Omega \vert$. However, a standard result on Sobolev spaces (see e.g. \cite[Section 1.1.3]{Mazja1985}) implies that for each fixed $\theta \in [0,\pi)$, $w_{x,\theta}$ is twice differentiable at $0$ for almost every $x \in \Omega$. We thus have the equality $\vert x \in \Omega \,:\, w_{x,\theta}''(0) \textnormal{ exists} \vert \,= \vert \Omega \vert$ for each fixed $\theta$, so integrating over $\theta$ gives 
\[
I = \int_0^\pi \left\vert x \in \Omega \,:\, w_{x,\theta}''(0) \textnormal{ exists} \right\vert \, \mathrm{d}\theta = \pi \vert \Omega \vert,
\]
a contradiction. \\
Now, take normal co-ordinates $(u,v)$ about some fixed point $x = \omega^{-1}(0,0)$, where $\omega$ is the co-ordinate chart map, so the Laplace-Beltrami operator at $x$ is given by the standard Laplacian operator $\partial^2_{uu} + \partial^2_{vv}$. The second derivative of $w(l_\theta(t))$ evaluated at $t=0$ is given by
\[
\partial^2_{tt} w(l_\theta(t)) \Bigr|_{t=0} = \cos^2(\theta) \partial^2_{uu}w(x) + \sin^2(\theta) \partial^2_{vv}w(x) + \sin(\theta)\cos(\theta) (\partial^2_{uv}w(x) + \partial^2_{vu}w(x)),
\]
and because $\partial^2_{uu} w(x) + \partial^2_{vv} w(x) \neq 0$ we have that $\partial^2_{tt} w(l_\theta(t)) \bigr|_{t=0}$ viewed as a function of $\theta$ is real analytic and not identically zero, so the set of values of $\theta$ such that $\partial^2_{tt} w(l_\theta(t)) \bigr|_{t=0} = 0$ must be null in $[0,\pi)$ (with respect to the one-dimensional Lebesgue measure). Thus, along almost every line $l_\theta = l_{x,\theta}$ of our foliation, $x$ must be an isolated zero of the restriction of $w$ to $l_\theta$: that is, for almost every $\theta \in [0,\pi)$ there exists $\varepsilon > 0$ such that $w(l_\theta(t)) \neq 0$ for $t \in (-\varepsilon,\varepsilon)\backslash \{ 0 \}$. This implies the convergence
\begin{equation}\label{eq:lebdiff}
\frac{\mu \left( B_\varepsilon(x) \cap \{ w = 0 \} \right)}{\mu \left( B_\varepsilon(x) \right)} \rightarrow 0 \textnormal{ as } \varepsilon \rightarrow 0.
\end{equation}
Indeed, assume for a contradiction that there exists some $c \in (0,1)$ and a sequence $\varepsilon_n \rightarrow 0$ such that
\begin{equation}\label{eq:lebdiff2}
\frac{\mu \left( B_{\varepsilon_n}(x) \cap \{ w = 0 \} \right)}{\mu \left( B_{\varepsilon_n}(x) \right)} \geq c.
\end{equation}
For each $n \in \mathbb{N}$, consider the set
\[
A_n = \{ \theta \in [0,\pi) \,:\, w(l_\theta(t)) \textnormal{ has a zero in }(-\varepsilon_n, \varepsilon_n)\backslash \{ 0 \} \}.
\]
By equation \eqref{eq:lebdiff2} the Lebesgue measure of $A_n$ is at least $c \pi$, and by definition we have the nesting of sets
\[
A_1 \supseteq A_2 \supseteq \cdots
\]
so we find that the set
\[
A := \bigcap_{n \in \mathbb{N}} A_n 
\]
has Lebesgue measure $\vert A \vert\, \geq c \pi > 0$. This contradicts the fact that $t = 0$ is an isolated zero of $w(l_\theta(t))$ for almost every $\theta \in [0,\pi)$, and so \eqref{eq:lebdiff} is proven. \\
We now consider the indicator function $\mathbbm{1}_{ \{ w = 0 \, \& \, H \neq 0 \} }$. By definition, this function is zero at all points $x \in \Sigma$ where either $w(x) \neq 0$ or $H(x) = 0$. At the remaining points, \eqref{eq:lebdiff} and the Lebesgue differentiation theorem imply that $\mathbbm{1}_{ \{ w = 0 \, \& \, H \neq 0 \} } = 0$ up to null sets. That is, $\mathbbm{1}_{ \{ w = 0 \, \& \, H \neq 0 \} } = 0$ almost everywhere on $\Sigma$, i.e. the set of zeroes of $w$ where $H$ does not vanish is null, which is the statement we wanted to prove. \\
If instead $\lambda = 0$, equation \eqref{eq:main1} reduces to $\Delta w + 2 Q w = 0$, and a result from \cite{ALESSANDRINI2012669} implies that all zeroes of $w$ must then have finite orders of vanishing. This in turn allows us to use a result of \cite{hardt1989nodal} and conclude that the zero set of $w$ is null, i.e. the condition that $H = 0$ almost everywhere on the set $\{w = 0\}$ with respect to $\mu$ is vacuously true.
\end{proof}

\noindent All that remains now is to prove Proposition \ref{prop:sufexist}. 
\begin{proof}[Proof of Proposition \ref{prop:sufexist}]
The proposition follows from $L^p$ approximation with no penalisation term included ($\sigma = 0$). The same calculations as in Section \ref{subsec:epsapprox} and in Lemma \ref{lem:epsconvimm} of Section \ref{subsec:invconv} show that for fixed $2 < p < \infty$, minimisers of $\mathcal{M}_p := \mathcal{M}_p^0$ exist. Taking the sequence $(f_p)$ of $\mathcal{M}_p$-minimal immersions, the same calculations again show that $(f_p)$ weakly converges in $W^{2,q}$ for all $q < \infty$ to a limiting immersion $f$. Both times here the assumptions of Proposition \ref{prop:sufexist} have been used to ensure the inequalities in Theorem \ref{thm:KLL} are satisfied. As in Corollary \ref{cor:hpconv}, we have the equality
\[
\mathcal{M}_\infty [f] = \lim_{p} \mathcal{M}_p [f_p]
\]
and we will now argue by contradiction that $f$ is a minimiser of $\mathcal{M}_\infty$ (and hence an $\infty$-Willmore sphere by Definition \ref{def:infWillmoresurface}). Supposing there exists another immersion $f_0$ such that $\mathcal{M}_\infty[f_0] < \mathcal{M}_\infty[f]$, by H\"{o}lder's inequality we have that
\[
\lim_p A^{-\frac{1}{p}} \mathcal{M}_p[f_0] \leq \mathcal{M}_\infty[f_0] < \mathcal{M}_\infty[f] = \lim_p A^{-\frac{1}{p}} \mathcal{M}_p [f_p],
\]
so for some $p \in \mathbb{N}$ the inequality
\[
\mathcal{M}_p[f_0] < \mathcal{M}_p [f_p]
\]
holds, which contradicts the fact that $f_p$ is minimal for $\mathcal{M}_p$.
\end{proof}





\noindent \textbf{Acknowledegements.} Ed Gallagher is grateful for being funded by a studentship from the EPSRC, project reference EP/V520305/1, studentship 2446338.

\newpage
\bibliographystyle{unsrt}
\bibliography{bibliography}

\end{document}